\documentclass[10pt,a4paper]{article}

\usepackage{amssymb}
\usepackage{amsfonts}
\usepackage{amsmath}
\usepackage{graphicx}

\newtheorem{theorem}{Theorem}

\newtheorem{corollary}[theorem]{Corollary}
\newtheorem{assumption}[theorem]{Assumption}
\newtheorem{definition}[theorem]{Definition}
\newtheorem{example}[theorem]{Example}

\newtheorem{lemma}[theorem]{Lemma}

\newtheorem{proposition}[theorem]{Proposition}
\newtheorem{remark}[theorem]{Remark}

\newenvironment{proof}[1][Proof]{\textbf{#1.} }{\hfill $\Box$}

\newcommand{\eps}{\varepsilon}
\renewcommand{\epsilon}{\eps}

\begin{document}
\title{Fast Diffusion Limit for Reaction-Diffusion Systems with Stochastic
Neumann Boundary Conditions}
\author{ Wael W. Mohammed$^{1}$ and Dirk Bl\"{o}mker$^{2}$ \\
$^{1}$Department of Mathematics, Faculty of \ Science, \\
Mansoura University, Egypt\\
{\small E-mail: wael.mohammed@mans.edu.eg}\\
$^{2}$Institut f\"{u}r Mathematik\\
Universit\"at Augsburg, Germany\\
{\small E-mail: dirk.bloemker@math.uni-augsburg.de } }
\date{\today }
\maketitle

\begin{abstract}
We consider a class of reaction-diffusion equations with a stochastic
perturbation on the boundary. We show that in the limit of fast diffusion,
one can rigorously approximate solutions of the system of PDEs with
stochastic Neumann boundary conditions by the solution of a suitable
stochastic/deterministic differential equation for the average concentration
that involves reactions only. An interesting effect occurs, if the noise on
the boundary does not change the averaging concentration, but is
sufficiently large. Then surprising additional effective reaction terms
appear.

We focus on systems with polynomial nonlinearities only and give
applications to the two dimensional nonlinear heat equation and the cubic
auto-catalytic reaction between two chemicals.
\end{abstract}



\section{ Introduction}

Stochastic partial differential equations (SPDEs) appear naturally as models
for dynamical systems with respect to random influences. Sometimes in a
complex physical system the noise has an impact not only on the bulk of the
system but on its physical boundary, too. This happens for instance in heat
transfer in a solid in contact with a fluid \cite{Langer}, chemical reactor
theory \cite{Lapidus}, colloid and interface chemistry \cite{Vold}, and the
air-sea interactions on the ocean surface \cite{Peixoto}.

Let $G$ be a bounded sufficiently smooth domain in $\mathbb{R}^{d}$ for $%
d\geq 1,$ which has a smooth boundary $\partial G$. We consider the
following system of stochastic reaction-diffusion equations for $n$ species
with respect to random Neumann boundary conditions%
\begin{eqnarray}
\partial _{t}u &=&\varepsilon ^{-2}\mathcal{A}u+\mathcal{F}(u),\ \ \ \text{%
for}\ t\geq 0,\ x\in G,  \notag \\
\frac{\partial u}{\partial \nu } &=&\sigma _{\epsilon }\partial _{t}W(t)\ \
\ \ \ \ \ \text{for}\ \ t\geq 0,\ x\in \partial G,  \label{eq1a} \\
u(0,x) &=&u_{0}(x)\ \ \ \ \ \ \ \text{for}\ \ x\in G,  \notag
\end{eqnarray}%
with%
\begin{equation*}
\mathcal{A}u=\left( 
\begin{array}{c}
\mathcal{A}_{1}u_{1} \\ 
\vdots \\ 
\mathcal{A}_{n}u_{n}%
\end{array}%
\right) ,\ \ \mathcal{F}(u)=\left( 
\begin{array}{c}
\mathcal{F}_{1}(u_{1},..,u_{n}) \\ 
\vdots \\ 
\mathcal{F}_{n}(u_{1},..,u_{n})%
\end{array}%
\right) \text{ \ and }W(t)=\left( 
\begin{array}{c}
W_{1}(t) \\ 
\vdots \\ 
W_{n}(t)%
\end{array}%
\right) ,
\end{equation*}%
where $\mathcal{A}$ is the diffusion term, the reaction terms $\mathcal{F}%
_{i}(u_{1},u_{2},....,u_{n})$ are polynomials of degree $m_{i}$, $W_{i}$ are
independent $Q$-Wiener process in $\mathcal{L}^{2}(\partial G)$, and $\frac{%
\partial u}{\partial \nu }$\ is the normal derivative of $u$ on $\partial G$%
. The assumption of independence is mainly for convenience of presentation,
as now some terms cancel and the technicalities are less involved.

Sowers \cite{Sowers} investigated multidimensional stochastic reaction
diffusion equation with Neumann boundary conditions and he showed that there
is a unique solution. Da Prato and Zabczyk \cite{PratoZabczyk2,
PratoZabczyk1} discussed the difference between the problems with Dirichlet
and Neumann boundary noises, while \cite{Dir1, Dir2} study random Dirichlet
boundary conditions. Other results are \cite{Mas, BD}.

An very interesting result is by Schnaubelt and Veraar \cite{mildweak}, where regularity of solutions
is studied. Furthermore, mild and weak solutions are shown to coincide.

Recently, Cerrai and Freidlin \cite{Cerrai} considered a class of stochastic
reaction-diffusion equations with Neumann boundary noise. Also, they showed
that when the diffusion rate is much larger than the rate of reaction, it is
possible to replace the SPDE by a suitable one-dimensional stochastic
differential equation. But their result only allowed for weak convergence of
the approximation without any order of the error.

Our aim is to establish rigorously error bounds results for the
fast-diffusion limit for the general class of PDEs with stochastic Neumann
boundary conditions given by (\ref{eq1a}). The error estimates are performed
in an $\mathcal{L}^{p}$-space setting, as we cannot expect solutions to (\ref%
{eq1a}) to be smooth. Especially, at the forced boundary the solution $u$ is
expected to be even unbounded, although it is smoother inside the domain.
See \cite{mildweak} or for Dirichlet boundary \cite{Dir1}.

We consider two cases. The second on is the relatively simple limit, where
the fast diffusion just disappears in the limit, while in the first case
large noise changes the limiting reaction equation. The reason for large
noise might be that both diffusion and noise are enhanced by stirring.

\textit{First case}: If the noise does not change the average ($W_c=0$) but
is sufficiently large ($\sigma _{\epsilon }=\varepsilon ^{-1}$), then the
solutions of Equation (\ref{eq1a}) are well approximated by 
\begin{equation}
u(t,x)\simeq b(t)+\mathcal{Z}^{s}(t,x)+\text{error},  \label{Aprox}
\end{equation}%
where $b(t)\in \mathbb{R}^{n}$ represents the average concentration of the
components of $u$ given in general formulation as a solution of 
\begin{equation}
\partial _{t}b(t)=\mathcal{F}(b(t))+\mathcal{G}(b(t)),  \label{ODE}
\end{equation}%
for some polynomial $\mathcal{G}$\ of degree less than or equal $m-2$
depending on the structure of the noise. The stochastic perturbation $%
\mathcal{Z}^{s}(t,x)$ is defined later in (\ref{E40}). It is an $\varepsilon$%
-dependent fast Ornstein-Uhlenbeck process (OU-process) corresponding to
white noise in the limit $\varepsilon \rightarrow 0$. The index $c$ denotes
the average (i.e., $v_{c}=|G|^{-1}\int_{G}vdx$ which is the projection onto
the constants).

The ODE $\partial _{t}b(t)=\mathcal{F}(b)$\ is the expected result, but due
to noise an additional term of noise induced effective reactions appears. We
illustrate our results using a relatively simple auto-catalytic reaction. For
the result presented we always need a square which averages to a constant in
the limit $\epsilon\to 0$. This is mainly, because we assumed independent
noise terms for each species. In contrast, if the noise terms are dependent,
then any reaction term could lead to an additional effective reaction term
in the limit.

\textit{Second case}: If $W_{c}\neq 0$ and $\sigma _{\epsilon }=1$, then the
solution of Equation (\ref{eq1a}) are well approximated by%
\begin{equation}
u(t,x)=b(t)+\text{error},  \label{Aprox1}
\end{equation}%
and $b$\ is the solution of stochastic ordinary differential equation%
\begin{equation}
\partial _{t}b(t)=\mathcal{F}(b(t))+\partial _{t}\tilde{\beta}(t),
\label{SODE}
\end{equation}%
for some Wiener process $\tilde{\beta}$ in $%
\mathbb{R}
^{n},$ which is essentially the projection of $W$ onto the dominant constant
modes, i.e. the direct impact of the noise on the average. This is the
somewhat expected result, where the reaction-diffusion equation under fast
diffusion is well approximated by the reaction ODE.

As an application of our results, we give some examples from physics
(nonlinear heat equation) and from chemistry (cubic auto-catalytic reaction
between two chemicals according to the rule $A+B\rightarrow 2B$). To
illustrate our results let us focus for a moment on the relatively simple
two dimensional nonlinear heat equation (also called Ginzburg-Landau or
Allen-Cahn), which is partly covered by the setting of \cite{Cerrai}, too. 
\begin{eqnarray}
\partial _{t}u &=&\varepsilon ^{-2}\Delta u+u-u^{3}\ \ \text{for}\ t\geq 0,\
x\in \left[ 0,1\right] ^{2},  \notag \\
\frac{\partial u}{\partial \nu } &=&\sigma _{\epsilon }\partial _{t}W(t)\ \
\ \ \ \ \ \text{for}\ \ t\geq 0,\ x\in \partial \left[ 0,1\right] ^{2}.
\label{Heat1}
\end{eqnarray}%
For the first case we suppose $W_{c}=0$ and $\sigma _{\epsilon }=\varepsilon
^{-1}$, and our main Theorem \ref{thm}\ states that the solution of (\ref%
{Heat1}) is well approximated by (\ref{Aprox}) and $b$ is the solution of%
\begin{equation*}
db=[(1-C_{\alpha ,\lambda })b-b^{3}]dt,
\end{equation*}%
where $C_{\alpha ,\lambda }$ is a constant depending on the noise intensity
parameters $\alpha _{i,k}$\ and the eigenvalues of the operator $\Delta $.

For the second case $W_{c}\neq 0$ and $\sigma _{\epsilon }=1$ our main
Theorem \ref{thm2} states that the solution of (\ref{Heat1}) is of the form (%
\ref{Aprox1}) and $b$ is the solution of%
\begin{equation*}
db=[b-b^{3}]dt+dB,
\end{equation*}%
where $B$\ is a $\mathbb{R}$-valued standard Brownian motion.

The main novelties of this paper are on one hand the explicit error estimate%
\textbf{\ }in terms of high moments of the error, as usually only weak
convergence is treated (see e.g. \cite{Cerrai}), and on the other hand the
observation that large mass-conservative noise has the potential to change
effective reaction equations in the limit of large diffusion.

The paper is organized as follows. Our assumptions and some definitions are
given in the next section. In Section $3$ we derive the fast-diffusion limit
with error terms and present the main theorem. Section $4$ gives bounds for
high non-dominant modes, while Section $5$ provides averaging results over
the fast OU-process. In Section $6$, we give the proof of the approximation
Theorem I and some examples from physics and chemistry as applications of
our results. Finally, we prove the approximation Theorem II and apply this
result to nonlinear heat equation and cubic auto-catalytic reaction between
two chemicals.

\section{Definition and Assumptions}

This section states the precise setting for (\ref{eq1a}) and summarizes all
assumptions necessary for our results. For the analysis we work in the
separable Hilbert space $\mathcal{L}^{2}(G)$ of square integrable functions$%
, $ where $G\subset \mathbb{R}^{d}$ is a bounded domain with sufficiently
smooth boundary $\partial G$ (e.g. Lipschitz), equipped with scalar product $%
\langle \cdot ,\cdot \rangle $ and norm $\Vert \cdot \Vert $.

\begin{definition}
\label{Def.Lin}Define for $i=1,2,....,n$ and diffusion constants $d_{i}>0$%
\begin{equation}
\mathcal{A}_{i}=d_{i}\Delta \text{ }  \label{E101}
\end{equation}%
with%
\begin{equation*}
D(\mathcal{A}_{i})=\left\{ u\in \mathcal{H}^{2}:\left. \partial _{\nu
}u\right\vert _{\partial G}=0\right\} ,
\end{equation*}%
where $\partial _{\nu }u$\ is the normal derivative of $u$ on $\partial G$.
\end{definition}

Let $\left\{ g_{k}\right\} _{k=1}^{\infty }$ be an orthonormal basis of
eigenfunctions of $\mathcal{A}_{i}$ in $\mathcal{L}^{2}(G)$. It is obviously
the same basis for all $i$ with corresponding eigenvalues $%
\{d_{i}\lambda_{k}\}_{k=0}^{\infty }$ depending on $i$ ({cf. Courant and
Hilbert }\cite{CH}). Also, let $\left\{ e_{k}\right\} _{k=1}^{n}$ be the
standard orthonormal basis of $\mathbb{R}^{n}$. Hence, $\left\{
g_{k}e_{i}\right\} $ for $k\in \mathbb{N}_{0}$ and $i=\{1,..,n\}$, is an
orthonormal basis of $\mathcal{A}$ in $\left[ \mathcal{L}^{2}(G)\right] ^{n}$
such that $\mathcal{A}\left( g_{k}e_{i}\right) =-d_{i}\lambda _{k}g_{k}e_{i}.
$

\begin{assumption}
\label{assg} We assume that for all $k\in\mathbb{N}$ 
\begin{equation*}
\|g_k\|_\infty \leq C \lambda_k^{\gamma_1} \quad\text{ for some }\gamma_1
\geq 0 \;. 
\end{equation*}
\end{assumption}

This is true in $\mathbb{R}^2$ for instance on squares, hexagons, and
triangles with $\gamma_1=0$, while the worst case is $\gamma_1=(d-1)/2$
realized for balls and spheres. See \cite{Grie}. This condition might be
relaxed, but we focused in examples mainly on cases with $\gamma_1=0$.

Define 
\begin{equation*}
\mathcal{N} :=\ker \mathcal{A} =\text{span}\{e_{1}g_{0},....,e_{n}g_{0}\}, 
\end{equation*}
where $g_{0}=\left\vert G\right\vert ^{-\frac{1}{2}}$ is a constant and $%
\lambda _{0}=0$. Define $S=\mathcal{N}^{\bot }$ to be the orthogonal
complement of $\mathcal{N}$ in$\ \left[ \mathcal{L}^{2}(G)\right] ^{n}.$
Denote by $P_{c}u=\frac{1}{\left\vert G\right\vert }\int_{G}udx$ the
projection onto $\mathcal{N}$ and define $P_{s}u:=(\mathcal{I}-P_{c})u$ for
the projection onto the orthogonal complement, where $\mathcal{I}$\ is the
identity operator on $\left[ \mathcal{L}^{2}(G)\right] ^{n}.$ We define $%
\mathcal{L}_{n}^{p}:=\left[ \mathcal{L}^{p}(G)\right] ^{n}$.

The operator $\mathcal{A}$ given by Definition \ref{Def.Lin} generates an
analytic semigroup $\{e^{t\mathcal{A}}\}_{t\geq 0}$ (cf. Dan Henry \cite%
{henry} or Pazy \cite{Pazy}), on $\mathcal{L}_{n}^{p}$ for all $p\geq 2$. It
has the following property: There is an $\omega >0$\ such that for all $t>0$
and all $u\in \mathcal{L}_{n}^{p}$%
\begin{equation}
\left\Vert e^{t\mathcal{A}}P_{s}u\right\Vert _{\mathcal{L}_{n}^{p}}\leq
Me^{-\omega t}\left\Vert P_{s}u\right\Vert _{\mathcal{L}_{n}^{p}},
\label{E3}
\end{equation}%
where $\omega $ depends in general on $d_{i}$.

Moreover, we obtain 
\begin{equation}
\left\Vert e^{t\mathcal{A}}u\right\Vert _{\mathcal{L}_{n}^{p}}\leq
M\left\Vert u\right\Vert _{\mathcal{L}_{n}^{p}}.  \label{E3d}
\end{equation}%
Also, we suppose

\begin{assumption}
\label{asscoeff}There is a constant $M>0$ such that for all $t>0$ and $u\in 
\mathcal{L}_{n}^{mp}$%
\begin{equation}
\left\Vert e^{t\mathcal{A}}u\right\Vert _{\mathcal{L}_{n}^{mp}}\leq
M(1+t^{-\alpha })\left\Vert u\right\Vert _{\mathcal{L}_{n}^{p}}  \label{E4d}
\end{equation}%
with $\alpha =\frac{d}{p}\left( \frac{m-1}{m}\right) \in (0,1).$
\end{assumption}

The previous assumption is needed for the existence of the solutions and
global bounds. Equation (\ref{E4d}) follows the Sobolev-embedding of $%
W^{\alpha ,p}$ into $\mathcal{L}^{mp}.$ The main assumption is that the
coefficient is between $0$ and $1$.

Immediate conclusion of Assumption \ref{asscoeff} and Equation (\ref{E3}) is%
\begin{equation}
\left\Vert e^{t\mathcal{A}}P_{s}u\right\Vert _{\mathcal{L}_{n}^{mp}}\leq
M(1+t^{-\alpha })e^{-\omega t}\left\Vert P_{s}u\right\Vert _{\mathcal{L}%
_{n}^{p}},  \label{E4dd}
\end{equation}%
where for simplicity we denote different constants $\omega ,$ $M$ by the same
name.

For the noise we suppose:

\begin{assumption}
\label{Wiener}Let $W=(W_{1},....,W_{n})$ be a collection of $n$ independent
Wiener process on an abstract probability space $(\Omega $, $F$, $\mathbb{P})
$ with a bounded covariance operator $Q_{i}:\mathcal{L}^{2}(\partial
G)\rightarrow \mathcal{L}^{2}(\partial G)$ defined by $Q_{i}f_{k}=\alpha
_{i,k}f_{k}$ for \ $i=1,2,..,n,$ where $\left( \alpha _{i,k}\right) _{k\in 
\mathbb{N}_{0}}$\ is a bounded sequence of real numbers and $\left(
f_{k}\right) _{k\in \mathbb{N}_{0}}$ be any orthonormal basis on $\mathcal{L}%
^{2}(\partial G)$ with $f_{0}\equiv $Constant.\ For $t\geq 0$ we can write $%
W_{i}(t)$ (cf. Da Prato and Zabczyk \cite{Prato}) as%
\begin{equation}
W_{i}(t)=\sum_{k\in \mathbb{N}_{0}}\alpha _{i,k}\beta _{i,k}(t)f_{k}\text{ \
for \ }i=1,2,..,n\text{,}  \label{E4}
\end{equation}%
where $\left( \beta _{i,k}\right) _{k\in \mathbb{N}_{0}}$ are independent,
standard Brownian motions in $\mathbb{R}$. Also, we assume by using the
orthonormal basis $g_{k}$ of $\mathcal{A}_{i}$ in $\mathcal{L}^{2}(G)$ that
for some small $\gamma \in (0,\frac{1}{2})$%
\begin{equation}
\sum_{k,\ell =1}^{\infty }(\lambda _{k}+\lambda _{\ell })^{2\gamma+2\gamma_1
-1}q_{k,\ell }^{i,i}<\infty \text{ for \ }i=1,2,..,n,  \label{e4a}
\end{equation}%
where the covariance $q_{k,\ell }^{i,j}$ is defined by%
\begin{equation}
q_{j,k}^{i,\ell }=\frac{1}{t}\mathbb{E}\left( \tilde{W}_{i,j}\tilde{%
(t)W_{\ell ,k}}(t)\right) =\left\{ 
\begin{array}{c}
0\text{ \ \ \ \ \ \ \ \ \ \ \ \ \ \ \ \ \ \ \ \ if }i\neq \ell , \\ 
\left\langle Q_{i}g_{j},g_{k}\right\rangle _{\mathcal{L}^{2}(\partial G)}%
\text{ \ \ if }i=\ell ,%
\end{array}%
\right.  \label{E10d}
\end{equation}%
with%
\begin{equation}
\tilde{W}_{i,j}=\left\langle W_{i},g_{j}\right\rangle _{\mathcal{L}%
^{2}(\partial G)}.  \label{E6c}
\end{equation}
\end{assumption}

For the nonlinearity we assume

\begin{assumption}
\label{Poly}The nonlinearity $\mathcal{F}$ is a polynomial of at most degree 
$m$. Thus for all $p\geq 1$\ it is bounded by%
\begin{equation}
\left\Vert \mathcal{F}(u)\right\Vert _{\mathcal{L}_{n}^{p}}\leq
C(1+\left\Vert u\right\Vert _{\mathcal{L}_{n}^{pm}}^{m})\text{ for all }u\in 
\mathcal{L}_{n}^{pm}.  \label{E3a}
\end{equation}%
where $m=\max (m_{1},.....,m_{n})$ and the $m_i$ are the degrees of the
polynomials $\mathcal{F}_i$.
\end{assumption}
The following assumption ensures, that the noise is mass-conserving and that
various series converge. This is used in Case 1 only.
\begin{assumption}
\label{W-C}Assume for \ $i=1,2,..,n$ that%
\begin{equation*}
\alpha _{i,0}=0\text{,}
\end{equation*}%
and for any $N\leq m$ and any $\ell\in\{1,\ldots,N\}^n$ 
\begin{equation}
\sum_{k_{1},k_{2},..,k_{N}=1}^{\infty }\Big(\frac{1}{\sum%
\limits_{i=1}^{N}d_{\ell _{i}}\lambda _{k_{i}}}\mathop{\textstyle \prod }%
\limits_{i=1}^{N}\lambda _{k_{i}}^{2\gamma_1-1}\emph{q}_{k_{i},k_{i}}^{\ell
_{i},\ell _{i}}\Big)^{\frac{1}{2}}<\infty .  \label{z10a}
\end{equation}
\end{assumption}
\begin{remark}
Condition (\ref{z10a}) for all $N\geq 1$, for example in case $\gamma_1=0$,  is implied by the
weaker condition%
\begin{equation*}
\sum_{k=1}^{\infty }\frac{\Big(\emph{q}_{k,k}^{i,i}\Big)^{\frac{1}{2}}}{%
\lambda _{k}^{\frac{1}{2}+\frac{1}{2m}}}<\infty .
\end{equation*}%
\end{remark}
We fix a universal $T_0>0$ that is the upper bound for all times involved.

The following two assumptions are used in the two cases separately. They
are usually lemmas that follows directly from the fact that $\mathcal{F}$\
is a polynomial. Note that $T_{1}$ in general depends on the initial
condition $b(0)$.
\begin{assumption}
\label{Amp} Let $b(t)$ in $\mathcal{N}$\ be the solution of (\ref{ODE}).
Suppose there is a stopping time $T_{1}\leq T_{0}$ and a constant $C>0,$
such that 
\begin{equation}
\sup_{\lbrack 0,T_{1}]}|b| \leq C.  \label{E4a}
\end{equation}
\end{assumption}
\begin{assumption}
\label{Amp1} Let $b(t)$ in $\mathcal{N}$\ be the solution of (\ref{SODE}).
Suppose there is a stopping time $T_{1}\leq T_{0}$ and $C>0,$ such that for
sufficiently large $\zeta \gg 1$ and for $\delta >0$ and $\kappa \in (0,%
\frac{1}{m+1})$%
\begin{equation}
\mathbb{P}\Big(\sup_{t\in \lbrack 0,T_{1}]}|b(t)| ^{m-1}\leq C\ln
(\varepsilon ^{-\frac{1}{\zeta }})\Big)\geq 1-\varepsilon ^{\delta \kappa }.
\label{E4b}
\end{equation}%
We remark that $\zeta $\ depends mainly on $T_{0}$ and $\kappa $\ (cf.
Section $8$).
\end{assumption}
For our result we rely on a cut off argument. We consider only solutions $u$
that are not too large, as given by the next definition.
\begin{definition}
\label{stopping time} For a mild solution $u$ of (\ref{eq1a}) we define for $%
\kappa \in (0,\frac{1}{m+1})$ the stopping time $\tau ^{\ast }$\ as%
\begin{equation}
\tau ^{\ast }:=T_{0}\wedge \inf \left\{ t>0:\left\Vert u\right\Vert _{%
\mathcal{L}_{n}^{2m}}>\varepsilon ^{-\kappa }\right\} .  \label{E5}
\end{equation}
\end{definition}
We give error estimates in terms of the following $\mathcal{O}$-notation.
\begin{definition}
\label{def:O} For a real-valued family of processes $\left\{ X_{\varepsilon
}(t)\right\} _{t\geq 0}$\ we say that $X_{\varepsilon }$\ is of order $%
f_{\varepsilon }$, i.e. $X_{\varepsilon }=\mathcal{O}(f_{\varepsilon })$, if
for every $p\geq 1$\ there exists a constant $C_{p}$\ such that%
\begin{equation}
\mathbb{E}\sup_{t\in \lbrack 0,\tau ^{\ast }]}\left\vert X_{\varepsilon
}(t)\right\vert ^{p}\leq C_{p}f_{\varepsilon }^{p}.  \label{E6}
\end{equation}%
We use also the analogous notation for time-independent random variables.
\end{definition}
\begin{definition}
(Multi-Index Notation) Let $\ell \in \mathbb{N}_{0}^{n},$ i.e. $\ell =(\ell
_{1},\ell _{2},......,\ell _{n})$ be a vector of nonnegative integers, $%
u=(u_{1},u_{2},...u_{n})$. Then we define%
\begin{equation*}
\vert \ell\vert =\sum_{i=1}^n\ell _{i}, \quad \ell ! = \prod_{i=1}^n \ell
_{1}!, \quad u^{\ell } =\prod_{i=1}^n u_{i}^{\ell _{i}}, \quad D^{\ell }
=\partial _{u_{1}}^{\ell _{1}}\partial _{u_{2}}^{\ell _{2}}.....\partial
_{u_{n}}^{\ell _{n}} 
\end{equation*}
\end{definition}
%
%
\section{Random boundary conditions}
\begin{definition}
(Neumann map)\label{Def1} The Neumann map $\mathcal{D}:\mathcal{L}%
^{2}(\partial G)\rightarrow \mathcal{H}^{\frac{3}{2}}(G)$\ is a continuous
linear operator. It is defined for $f\in \mathcal{L}^{2}(\partial G)$ as the
solution $\mathcal{D}f$\ of%
\begin{equation*}
(1-\Delta \mathcal{)D}f=0\ \ \text{and}\ \ \partial _{\nu }\left( \mathcal{D}%
f\right) =f\text{ }.
\end{equation*}
\end{definition}
With a slight abuse of notation, we also denote by $\mathcal{D}$ the
extension from $\mathcal{L}_{n}^{2}(\partial G)$ to $\left[ \mathcal{H}^{%
\frac{3}{2}}(G)\right] ^{n}$.
\begin{definition}
Define the stochastic convolution $\mathcal{Z}(t)$ as%
\begin{equation}
\mathcal{Z}(t)=\sigma _{\epsilon }(1-\Delta )\int_{0}^{t}e^{\varepsilon
^{-2}(t-s)\mathcal{A}}\mathcal{D}dW(s).  \label{E6a}
\end{equation}
\end{definition}
The next lemma expands the stochastic convolution $\mathcal{Z}$ as a
Fourier series.
\begin{lemma}
Under Assumption \ref{Wiener} let $\mathcal{Z}$ be the stochastic
convolution defined in (\ref{E6a}), then (with  $\tilde{W}_{i,j}$ defined in (\ref{E6c}))
\begin{equation}
\mathcal{Z}(t)=\sigma _{\epsilon }\sum_{i=1}^{n}\sum_{j=0}^{\infty
}\int_{0}^{t}e^{-\varepsilon ^{-2}d_{i}(t-s)\lambda _{j}}d\tilde{W}%
_{i,j}(s)g_{j} \cdot e_{i}\;.  \label{E6b}
\end{equation}
\end{lemma}
\begin{proof}
Writing $\mathcal{Z}(t)$ in Fourier expansion yields
\begin{equation*}
\mathcal{Z}(t)=\sum_{i=1}^{n}\sum_{j=0}^{\infty }\left\langle \mathcal{Z}%
(t),e_{i}g_{j}\right\rangle g_{j}\cdot e_{i}.
\end{equation*}%
Using Equation (\ref{E6a})%
\begin{eqnarray*}
\left\langle \mathcal{Z}(t),e_{i}g_{j}\right\rangle _{\mathcal{L}%
_{n}^{2}(G)} &=&\left\langle \mathcal{Z}_{i}(t),g_{j}\right\rangle _{%
\mathcal{L}^{2}(G)} \\
&=&\Big\langle \sigma _{\epsilon }(1-\Delta )\int_{0}^{t}e^{\varepsilon
^{-2}(t-s)d_{i}\Delta }\mathcal{D}dW_{i}(s),g_{j}\Big\rangle _{\mathcal{L}%
^{2}(G)} \\
&=&\sigma _{\epsilon }\int_{0}^{t}e^{-\varepsilon ^{-2}d_{i}(t-s)\lambda
_{j}}\left\langle \mathcal{D}dW_{i}(s),(1-\Delta )g_{j}\right\rangle _{%
\mathcal{L}^{2}(G)} \\
&=&\sigma _{\epsilon }\int_{0}^{t}e^{-\varepsilon ^{-2}d_{i}(t-s)\lambda
_{j}}\{\left\langle \mathcal{D}dW_{i}(s),g_{j}\right\rangle _{\mathcal{L}%
^{2}(G)}-\left\langle \mathcal{D}dW(s),\Delta g_{j}\right\rangle _{\mathcal{L%
}^{2}(G)}\} \\
&=&\sigma _{\epsilon }\int_{0}^{t}e^{-\varepsilon ^{-2}d_{i}(t-s)\lambda
_{j}}\{\left\langle (1-\Delta )\mathcal{D}dW_{i}(s),g_{j}\right\rangle _{%
\mathcal{L}^{2}(G)} \\
&&\ \ \ \ \ \ \ \ \ \ \ \ \ \ \ \ \ \ \ \ \ \ \ \ \ \ \ \ \ +\left\langle
\partial _{\nu }\mathcal{D}dW_{i}(s),g_{j}\right\rangle _{\mathcal{L}%
_{n}^{2}(\partial G)}\} \\
&=&\sigma _{\epsilon }\int_{0}^{t}e^{-\varepsilon ^{-2}d_{i}(t-s)\lambda
_{j}}\left\langle dW_{i}(s),g_{j}\right\rangle _{\mathcal{L}^{2}(G)},
\end{eqnarray*}%
where we used Gauss--Green formula and Definition \ref{Def1}. Hence%
\begin{equation*}
\mathcal{Z}(t)=\sigma _{\epsilon }\sum_{i=1}^{n}\sum_{j=0}^{\infty
}\int_{0}^{t}e^{-\varepsilon ^{-2}d_{i}(t-s)\lambda _{j}}d\tilde{W}%
_{i,j}(s)g_{j}\cdot e_{i}\text{.}
\end{equation*}
It is easy to check, that this series converges in $\mathcal{L}_n^2$.
\end{proof}
%
%
\section{Limiting equation and main theorem}
%
%
In this section we derive formally the limiting equation for (\ref{eq1a})
and we state without proof the main theorem of this paper. First, let us
define the mild solution of Equation (\ref{eq1a}) according to \cite%
{PratoZabczyk2, PratoZabczyk1} as follows
\begin{definition}
For any fixed $\varepsilon >0$, we call a $\mathcal{L}_{n}^{p}$-valued
stochastic process $u$\ a mild solution of\ (\ref{eq1a}) in $\mathcal{L}%
_{n}^{p}$ if for all $t>0$ up to a positive stopping time 
\begin{eqnarray}
u(t) &=&e^{\varepsilon ^{-2}t\mathcal{A}}u(0)+\int_{0}^{t}e^{\varepsilon
^{-2}(t-s)\mathcal{A}}\mathcal{F}(u(s))ds + \mathcal{Z}(t).   \label{E7a}
\end{eqnarray}
\end{definition}
Because we are working with a locally
Lipschitz nonlinearity,
under Assumption \ref{asscoeff}, the existence and uniqueness of solutions is standard, 
once $\mathcal{Z}$ is sufficiently regular. See
e.g. \cite{Prato} and \cite{PratoZabczyk1}.

We can rewrite Equation (\ref{E7a}) by using Equation (\ref{E6b}) as%
\begin{eqnarray}
u(t) &=&e^{\varepsilon ^{-2}t\mathcal{A}}u(0)+\int_{0}^{t}e^{\varepsilon
^{-2}(t-s)\mathcal{A}}\mathcal{F}(u(s))ds  \notag \\
&&+\sigma _{\epsilon }\sum_{i=1}^{n}\sum_{j=0}^{\infty
}\int_{0}^{t}e^{-\varepsilon ^{-2}d_{i}(t-s)\lambda _{j}}d\tilde{W}%
_{i,j}(s)g_{j}\cdot e_{i}\text{,}  \label{E7b}
\end{eqnarray}%
with $\tilde{W}_{i,j}$ defined in (\ref{E6c}).

Now, let us discuss two cases depending on $\sigma _{\epsilon }$ and $\alpha
_{i,0}$ for $i=1,...,n.$
%
%
\subsection{First case: $\protect\sigma _{\protect\epsilon }=\protect%
\varepsilon ^{-1}$ and $\protect\alpha _{i,0}=0$ for $i=1,..,n$}
%
%
In this case Equation (\ref{E7b}) takes the form%
\begin{equation}
u(t)=e^{\varepsilon ^{-2}t\mathcal{A}}u(0)+\int_{0}^{t}e^{\varepsilon
^{-2}(t-s)\mathcal{A}}\mathcal{F}(u(s))ds+\mathcal{Z}^{s}(t)\text{,}
\label{E7c}
\end{equation}%
where%
\begin{equation}
\mathcal{Z}^{s}(t)=\sum_{i=1}^{n}\mathcal{Z}_{i}(t)e_{i}:=\sum_{i=1}^{n}%
\sum_{j=1}^{\infty }\mathcal{Z}_{i,j}(t)g_{j}\cdot e_{i}\text{,}  \label{E40}
\end{equation}%
with%
\begin{equation}
\mathcal{Z}_{i}(t)=\sum_{j=1}^{\infty }\mathcal{Z}_{i,j}(t)g_{j}\text{ \ for 
}i=1,2....,n,  \label{E40b}
\end{equation}%
and%
\begin{equation}
\mathcal{Z}_{i,j}(t)=\varepsilon ^{-1}\int_{0}^{t}e^{-\varepsilon
^{-2}d_{i}(t-s)\lambda _{j}}d\tilde{W}_{i,j}(s).  \label{E40a}
\end{equation}%
In order to derive the limiting equation, we split the solution $u$ into 
\begin{equation}
u(t,x)=a(t)+\psi (t,x),  \label{E7}
\end{equation}
with $a\in \mathcal{N}$ and $\psi \in \mathcal{S}$. Plugging (\ref{E7}) into
(\ref{E7c}) and projecting everything onto $\mathcal{N}$ and $\mathcal{S}$ 
we obtain (with $\mathcal{F}^{c}=P_{c}\mathcal{F}$ and $\mathcal{F}^{s}=P_{s}\mathcal{F%
}$)
\begin{equation}
a(t)=a(0)+\int_{0}^{t}\mathcal{F}^{c}(a+\psi )ds,  \label{E8}
\end{equation}%
and
\begin{equation}
\psi (t)=e^{\varepsilon ^{-2}t\mathcal{A}}\psi
(0)+\int_{0}^{t}e^{\varepsilon ^{-2}(t-\tau )\mathcal{A}_{s}}\mathcal{F}%
^{s}(a+\psi )d\tau +\mathcal{Z}^{s}(t).  \label{E9}
\end{equation}
Formally, we see later (cf.\ Lemma \ref{Lemma1}) that 
 $\psi $ is well approximated by the fast Ornstein-Uhlenbeck process $\mathcal{Z}%
^{s}$. Thus, we can eliminate $\psi $ in (\ref{E8}) by explicitly
averaging over the fast modes.

Now the first main result of this paper is:

\begin{theorem}
\label{thm}(Approximation I) Under Assumptions \ref{assg}, \ref{Wiener}, \ref%
{W-C}, \ref{Poly}, and \ref{Amp}, let $u$ be a solution of (\ref{eq1a}) with
splitting $u=a+\psi $ defined in (\ref{E7}) with the initial condition $%
u(0)=a(0)+\psi (0)$ with $a(0)\in \mathcal{N}$ and $\psi (0)\in S$\ where $%
a(0)$ and $\psi (0)$ are of order one, and $b$ is a solution of (\ref{ODE})
with $b(0)=a(0)$. Then for all $p>0$ and all $\kappa \in (0,\frac{1}{2m+1})$%
, there exist a constant $C>0$ such that%
\begin{equation}
\mathbb{P}\Big(\sup_{t\in [ 0,T_{1}\wedge \tau ^{\ast }] }\Big\|u(t)-b(t)-%
\mathcal{Q}(t)\Big\|_{\mathcal{L}_{n}^{p}}>\varepsilon ^{1-2m\kappa -\kappa }%
\Big)\leq C\varepsilon ^{p},  \label{eq16}
\end{equation}%
where with fast OU-process $\mathcal{Z}^{s}$ defined in (\ref{E40})
\begin{equation}
\mathcal{Q}(t)=e^{\varepsilon ^{-2}t\mathcal{A}_{s}}\psi (0)+\mathcal{Z}%
^{s}(t). \label{eq17}
\end{equation}%
\end{theorem}
We see that the first part of (\ref{eq17}) depending on the initial
condition decays exponentially fast on the time-scale of order $\mathcal{O}%
(\varepsilon ^{2})$.
\begin{corollary}
\label{coll5} If in the previous theorem we additionally assume that
Assumption \ref{asscoeff} holds and $\left\Vert \psi (0)\right\Vert _{%
\mathcal{L}_{n}^{mp}}\leq C$ for some $C>0$, then we can replace $%
T_{1}\wedge \tau ^{\ast }$ in (\ref{eq16}) by $T_{1}$.
\end{corollary}
\begin{remark}
In case of Corollary \ref{coll5} we can bound the error even in $\mathcal{L}%
_{n}^{pm}.$
\end{remark}
%
%
\subsection{Second case $\protect\sigma _{\protect\epsilon }=1$ and $\protect%
\alpha _{i,0}\neq 0$ for $i=1,..,n$}
%
%
In this case (\ref{E7b}) takes the form%
\begin{eqnarray}
u(t) &=&e^{\varepsilon ^{-2}t\mathcal{A}}u(0)+\int_{0}^{t}e^{\varepsilon
^{-2}(t-s)\mathcal{A}}\mathcal{F}(u(s))ds  \notag \\
&&+\sum_{i=1}^{n}\sum_{j=0}^{\infty }\int_{0}^{t}e^{-\varepsilon
^{-2}d_{i}(t-s)\lambda _{j}}d\tilde{W}_{i,j}(s)g_{j}\cdot e_{i}\text{.}
\label{E1}
\end{eqnarray}%
Again (cf.\  (\ref{E7})) 
we split the solution $u$ into
$u(t,x)=a(t)+\varepsilon \psi (t,x).$
Plugging (\ref{E7}) into (\ref{E1}) and projecting everything onto $%
\mathcal{N}$ and $\mathcal{S}$ yields%
\begin{equation}
a(t)=a(0)+\int_{0}^{t}\mathcal{F}^{c}(a+\varepsilon \psi )ds+\sum_{i=1}^{n}%
\tilde{W}_{i,0}(t) g_{0} \cdot e_{i},  \label{E10a}
\end{equation}
and
\begin{equation}
\psi (t)=e^{\varepsilon ^{-2}t\mathcal{A}}\psi (0)+\frac{1}{\varepsilon }%
\int_{0}^{t}e^{\varepsilon ^{-2}(t-\tau )\mathcal{A}_{s}}\mathcal{F}%
^{s}(a+\varepsilon \psi )d\tau +\mathcal{Z}^{s}(t),  \label{E10}
\end{equation}
where $\mathcal{Z}^{s}(t)$ was defined in (\ref{E40}). We write 
(\ref{E10a}) as%
\begin{equation*}
a_{i}(t)=a_{i}(0)+\int_{0}^{t}\mathcal{F}_{i}^{c}(a+\varepsilon \psi )ds+%
\tilde{W}_{i,0}(t)g_{0}\text{ \ for }i=1,2...,n.
\end{equation*}
Now, applying Taylor's expansion to the function $\mathcal{F}_{i}^{c}:%
\mathcal{L}^{2}(G)\rightarrow \mathbb{R},$\ yields the following stochastic limiting
equation with error%
\begin{equation}
a_{i}(t)=a_{i}(0)+\int_{0}^{t}\mathcal{F}_{i}(a)ds+\tilde{W}%
_{i,0}(t)g_{0}+R_{i}^{(2)}(t),  \label{Ampl2}
\end{equation}%
where%
\begin{equation}
R_i^{(2)}(t)=\sum_{|\ell| \geq 1}P_{c}\int_{0}^{t}\frac{D^{\ell }\mathcal{F}%
_{i}(a)}{\ell !}\mathcal{(}\varepsilon \psi )^{\ell }d\tau =\mathcal{O}%
(\varepsilon ^{1-}).  \label{E11b}
\end{equation}%
The second main result of this paper is:
\begin{theorem}
\label{thm2}(Approximation II) Under Assumptions \ref{assg}, \ref{Wiener}, %
\ref{Poly} and \ref{Amp1}, let $u$ be a solution of (\ref{eq1a}) with
splitting $u=a+\varepsilon \psi $ defined in (\ref{E7}) with the initial
condition $u(0)=a(0)+\varepsilon \psi (0)$ with $a(0)\in \mathcal{N}$ and $%
\psi (0)\in S $\ where $a(0)$ and $\psi (0)$ are of order one, and $b$ is a
solution of (\ref{SODE}) with $b(0)=a(0)$. Then for all $p>0,\ $for
sufficiently large $\zeta \gg 1$ and all $\kappa \in (0,\frac{1}{m+2})$,
there exists $C>0$ such that%
\begin{equation}
\mathbb{P}\Big(\sup_{t\in [ 0,T_{1}\wedge \tau ^{\ast }] }\|u(t)-b(t)\|_{%
\mathcal{L}_{n}^{p}}>\varepsilon ^{1-(m+2)\kappa }\Big)\leq C\varepsilon
^{\delta \kappa }.  \label{eq16a}
\end{equation}
\end{theorem}
In our examples if we assume $\mathbb{E}\exp\{ c\delta \vert b(0)\vert
^{m-1}\} \leq C$ for some suitable $c>0$ and for one $\delta >0,$\ then
Assumption \ref{Amp1}\ is true. See Section $8.1$.
\begin{corollary}
\label{coll5a} If in the previous theorem additionally Assumption \ref%
{asscoeff} holds and $\Vert \psi (0)\Vert _{\mathcal{L}_{n}^{mp}}\leq C$ for 
$C>0$, then we can replace $T_{1}\wedge \tau ^{\ast }$ in (\ref{eq16a}) by $%
T_{1}$.
\end{corollary}
The sufficiently large $\zeta$ depends mainly on $\kappa$ and $T_0$.


\section{Bounds for the high modes}


Let us summarize Equations (\ref{E9}) and (\ref{E10}) for $\rho\in\{0,1\}$ by 
\begin{equation}
\psi (t)=e^{\varepsilon ^{-2}t\mathcal{A}}\psi (0)+\varepsilon ^{-\rho
}\int_{0}^{t}e^{\varepsilon ^{-2}(t-\tau )\mathcal{A}_{s}}\mathcal{F}%
^{s}(a+\varepsilon \psi )d\tau +\mathcal{Z}^{s}(t).  \label{E200}
\end{equation}
In the first lemma of this section, we see that $\psi$
is well approximated by the fast Ornstein-Uhlenbeck process 
$\mathcal{Z}^{s}$ (cf. (\ref{E40})).
\begin{lemma}
\label{Lemma1}Under Assumption \ref{Poly}, there is a constant $C>0$ such
that for $p\geq 1$ and $\kappa >0$ from the definition of $\tau ^{\ast }$%
\begin{equation}
\mathbb{E}\sup_{t\in [ 0,\tau ^{\ast }] }\Big\|\psi (t)-e^{\varepsilon ^{-2}t%
\mathcal{A}}\psi (0)-\mathcal{Z}^{s}(t)\Big\|_{\mathcal{L}_{n}^{p}}^{p}\leq
C\varepsilon ^{2p-p\rho -mp\kappa }.  \label{E13}
\end{equation}
\end{lemma}
\begin{proof}
From (\ref{E200}) using semigroup estimates and Assumption \ref{Poly}
we obtain%
\begin{eqnarray*}
\Big\|\psi (t)-e^{\varepsilon ^{-2}t\mathcal{A}}\psi (0)-\mathcal{Z}^{s}(t)%
\Big\|_{\mathcal{L}_{n}^{p}} &=&\frac{1}{\varepsilon ^{\rho }}\Big\|%
\int_{0}^{t}e^{\varepsilon ^{-2}\mathcal{A}_{s}(T-\tau )}\mathcal{F}%
^{s}(u)d\tau \Big\|_{\mathcal{L}_{n}^{p}} \\
&\leq &C\varepsilon ^{-\rho }\sup_{\tau \in [ 0,\tau ^{\ast }] }\left\Vert 
\mathcal{F}^{s}(u)\right\Vert _{\mathcal{L}_{n}^{p}}\int_{0}^{t}e^{-%
\varepsilon ^{-2}\omega (t-\tau )}d\tau \\
&\leq &C\varepsilon ^{2-\rho }\sup_{\tau \in [ 0,\tau ^{\ast }]
}(1+\left\Vert u\right\Vert _{\mathcal{L}_{n}^{pm}}^{m}) 
\leq C\varepsilon ^{2-\rho -m\kappa }.
\end{eqnarray*}
\end{proof}
\begin{lemma}
\label{Lemma2} Under Assumptions \ref{assg} and \ref{Wiener}, for every $%
\kappa _{0}>0$ and $p\geq 1$ there is a constant $C,$ depending on $p,\
\alpha _{k},$ $\lambda _{k},$\ $\kappa _{0}$ and $T_{0},$ such that%
\begin{equation}
\mathbb{E}\sup_{t\in [ 0,T_0] }\left\Vert \mathcal{Z}^{s}(t)\right\Vert _{%
\mathcal{L}_{n}^{p}}^{p}\leq C\varepsilon ^{-\kappa _{0}},  \label{eq19a}
\end{equation}%
where $\mathcal{Z}^{s}(t)$ was defined in\ (\ref{E40}).
\end{lemma}
\begin{proof}
We use the celebrated factorization method introduced in \cite{Prato} to
prove the bound on $\mathcal{Z}^{s}(t)=\sum_{i=1}^{n}\mathcal{Z}_{i}(t)e_{i}$%
, which is based on the following elementary identity%
\begin{equation}
\int_{\sigma }^{t}(t-r)^{\gamma -1}(r-\sigma )^{-\gamma }dr=\frac{\pi }{\sin
(\pi \gamma )}\text{ for }0\leq r\leq t,\text{ }0<\gamma <1.  \label{eq50}
\end{equation}
Fix $\gamma \in \left( 0,\frac{1}{2}\right) .$ To prove (\ref{eq19a}), it is
enough to bound $\mathcal{Z}_{i}$ for $i=1,\ldots n$. We obtain from
Equation (\ref{E40b}) that 
\begin{equation}
\mathcal{Z}_{i}(t)=\sum_{j=1}^{\infty }\varepsilon
^{-1}\int_{0}^{t}e^{-\varepsilon ^{-2}d_{i}(t-s)\lambda _{j}}d\tilde{W}%
_{i,j}(s)g_{j}=\varepsilon ^{-1}\int_{0}^{t}e^{\varepsilon ^{-2}(t-s)%
\mathcal{A}_{i}}d\tilde{W}_{i}(s),  \label{eq50b}
\end{equation}
where 
\begin{equation*}
\tilde{W}_{i}(t)=\sum_{j=1}^{\infty }\tilde{W}_{i,j}(s)g_{j}\text{ \ for }%
i=1,2,\ldots,n.
\end{equation*}
Using Identity (\ref{eq50}) with Equation (\ref{eq50b}), we obtain:%
\begin{equation*}
\mathcal{Z}_{i}(t)=C_{\gamma }\varepsilon ^{-1}\int_{0}^{t}e^{\varepsilon
^{-2}(t-\sigma )\mathcal{A}_{i}}\left[ \int_{\sigma }^{t}(t-s)^{\gamma
-1}(s-\sigma )^{-\gamma }dr\right] d\tilde{W}_{i}(\sigma ).
\end{equation*}%
From the stochastic Fubini theorem, we obtain%
\begin{eqnarray}
\mathcal{Z}_{i}(t) &=&C_{\gamma }\varepsilon ^{-1}\int_{0}^{t}e^{\varepsilon
^{-2}(t-s)\mathcal{A}_{i}}(t-s)^{\gamma -1}y_{i}(s)ds,  \label{eq50a}
\end{eqnarray}%
where%
\begin{eqnarray}
y_{i}(s) &=&\int_{0}^{s}e^{\varepsilon ^{-2}(s-\sigma )\mathcal{A}%
_{i}}(s-\sigma )^{-\gamma }d\tilde{W}_{i}(\sigma )  \notag \\
&=&\sum_{j=1}^{\infty }\int_{0}^{s}e^{-\varepsilon ^{-2}d_{i}(s-\sigma
)\lambda _{j}}(s-\sigma )^{-\gamma }d\tilde{W}_{i,j}(\sigma )g_{j}.
\label{eq50h}
\end{eqnarray}%
Taking $\left\Vert \cdot \right\Vert _{\mathcal{L}_{n}^{p}}^{p}$ on both
sides of (\ref{eq50a}) and using (\ref{E3}), we obtain 
\begin{equation*}
\Vert \mathcal{Z}_{i}(t)\Vert _{\mathcal{L}^{p}}^{p} \leq C_{\gamma
}^{p}\varepsilon ^{-p}\Big( \int_{0}^{t}e^{-\varepsilon ^{-2}(t-s)\omega
}(t-s)^{\gamma -1}\left\Vert y_{i}(s)\right\Vert _{\mathcal{L}^{p}}ds\Big) %
^{p}.
\end{equation*}%
Using H\"{o}lder inequality with $\frac{1}{p}+\frac{1}{q}=1$ for
sufficiently large $p$ implies%
\begin{eqnarray*}
\Vert \mathcal{Z}_{i}(t) \Vert _{\mathcal{L}^{p}}^{p} &\leq &C_{\gamma
}^{p}\varepsilon ^{-p}\Big( \int_{0}^{t}e^{-\varepsilon ^{-2}(t-s)\omega
}(t-s)^{q\gamma -q}ds\Big) ^{\frac{p}{q}}\cdot \int_{0}^{t}\left\Vert
y_{i}(s)\right\Vert _{\mathcal{L}^{p}}^{p}ds \\
&\leq &C\varepsilon ^{-2+2p(\gamma -\frac{1}{2})} \int_{0}^{t}\Vert
y_{i}(s)\Vert _{\mathcal{L}^{p}}^{p}ds .
\end{eqnarray*}%
Taking supremum after expectation, yields%
\begin{equation}
\mathbb{E}\sup_{t\in [ 0,T_0] }\left\Vert \mathcal{Z}_{i}(t)\right\Vert _{%
\mathcal{L}^{p}}^{p}\leq C\varepsilon ^{-2+2p(\gamma -\frac{1}{2}%
)}\cdot \int_{0}^{T_{0}}\mathbb{E}\left\Vert y_{i}(s)\right\Vert _{\mathcal{L%
}^{p}}^{p}ds .  \label{eq54}
\end{equation}%
Now, we bound $\mathbb{E}\left\Vert y_{i}(s)\right\Vert _{\mathcal{L}%
^{p}}^{p}.$ By Gaussianity%
\begin{equation*}
\mathbb{E}\left\Vert y_{i}(s)\right\Vert _{\mathcal{L}^{p}}^{p}=\mathbb{E}%
\int_{D}\vert y_{i}(s,x)\vert ^{p}dx\leq C_{p}\Big( \int_{D}\mathbb{E}\vert
y_{i}(s,x)\vert ^{2}\Big) ^{\frac{p}{2}} dx.  \label{eq51d}
\end{equation*}%
Hence by Definition of $y_{i}$ (\ref{eq50h})%
\begin{eqnarray*}
\mathbb{E}\left\vert y_{i}(s,x)\right\vert ^{2} &=&\mathbb{E}\big|%
\sum_{j=1}^{\infty }\int_{0}^{s}e^{-\varepsilon ^{-2}d_{i}(s-\sigma )\lambda
_{j}}(s-\sigma )^{-\gamma }d\tilde{W}_{i,j}(\sigma )g_j(x)\big|^{2}
\label{eq51} \\
&=&C\sum_{j,k=1}^{\infty }q_{j,k}^{i,i}\int_{0}^{s}e^{-\varepsilon
^{-2}d_{i}(s-\sigma )(\lambda _{j}+\lambda _{k})}(s-\sigma )^{-2\gamma
}d\sigma g_j(x)g_k(x),  \notag
\end{eqnarray*}%
where we used the definition of covariance operator (\ref{E10d}). Hence,
using the bounds on $g_j$%
\begin{eqnarray}
\mathbb{E}\left\vert y_{i}(s)\right\vert ^{2} &\leq &C\varepsilon
^{2-4\gamma }\sum_{j,k=1}^{\infty }(\lambda _{j}+\lambda
_{k})^{2\gamma+2\gamma_1 -1}q_{j,k}^{i,i} \leq C\varepsilon ^{2-4\gamma },
\label{eq52}
\end{eqnarray}%
where we used (\ref{e4a}). Thus%
\begin{equation}
\sup_{t\in[0,T_0]} \mathbb{E}\left\Vert y_{i}(s)\right\Vert _{\mathcal{L}%
^{p}}^{p}\leq C\varepsilon ^{p-2p\gamma }\text{ }.  \label{eq53a}
\end{equation}%
Now, returning to Equation (\ref{eq54}) and using Equation (\ref{eq53a}),
yields%
\begin{equation*}
\mathbb{E}\sup_{t\in [ 0,T_0] }\left\Vert \mathcal{Z}_{i}(t)\right\Vert _{%
\mathcal{L}_{n}^{p}}^{p}\leq C\varepsilon ^{-2}.
\end{equation*}%
We finish the proof by using H\"{o}lder inequality to derive for all $p>1$
and sufficiently large $q>\frac{2}{\kappa _{0}}$%
\begin{equation*}
\mathbb{E}\sup_{t\in [ 0,T_0] }\left\Vert \mathcal{Z}_{i}(t)\right\Vert _{%
\mathcal{L}_{n}^{p}}^{p}\leq \Big( \mathbb{E}\sup_{t\in [ 0,T_0] }\left\Vert 
\mathcal{Z}_{i}(t)\right\Vert _{\mathcal{L}_{n}^{p}}^{pq}\Big)^{\frac{1}{q}%
}\leq C\varepsilon ^{-\kappa _{0}}.
\end{equation*}
\end{proof}

The following corollary states that $\psi (t)$ is with high probability much
smaller than $\varepsilon ^{-\kappa }$ as assumed the Definition \ref%
{stopping time}\ for $t\leq \tau ^{\ast }$. We show later $\tau ^{\ast }\geq
T_{0}$ with high probability.
\begin{corollary}
\label{Cor1} Under the assumptions of Lemmas \ref{Lemma1} and \ref{Lemma2},
if $\psi (0)=\mathcal{O}(1)$, then for $p>0$ and $\rho =0$ or $1$\ there
exist a constant $C>0$ such that for $\kappa _{0}\leq \kappa $ 
\begin{equation}
\mathbb{E}\sup_{t\in [ 0,\tau ^{\ast }] }\left\Vert \psi (t)\right\Vert _{%
\mathcal{L}_{n}^{p}}^{p}\leq C\varepsilon ^{-\kappa _{0}}\text{ }.
\label{e22}
\end{equation}
\end{corollary}
\begin{proof}
By triangle inequality and Lemma \ref{Lemma2}, we obtain from (\ref{E13})%
\begin{equation*}
\mathbb{E}\sup_{t\in [ 0,\tau ^{\ast }] }\left\Vert \psi (t)\right\Vert _{%
\mathcal{L}_{n}^{p}}^{p}\leq C+C\varepsilon ^{2p-p\rho -mp\kappa
}+C\varepsilon ^{-\kappa _{0}}\text{ },
\end{equation*}%
which implies (\ref{e22}) for $\kappa <\frac{2-\rho }{m}$.
\end{proof}

Let us now state a result similar to averaging. When we integrate over the
fast decaying contribution of the initial condition in $\psi $, then this
leads to terms of order $\mathcal{O}(\varepsilon ^{2})$.
\begin{lemma}
\label{Lemma3}For $q\geq 1$ there exists a constant $C>0$ such that%
\begin{equation*}
\int_{0}^{t}\left\Vert e^{\tau \varepsilon ^{-2}\mathcal{A}_{s}}\psi
(0)\right\Vert _{\mathcal{L}_{n}^{p}}^{q}d\tau \leq C\varepsilon
^{2}\left\Vert \psi (0)\right\Vert _{\mathcal{L}_{n}^{p}}^{q}\text{ for }%
\psi (0)\in \mathcal{L}_{n}^{p}.
\end{equation*}
\end{lemma}
\begin{proof}
Using (\ref{E3}) we obtain%
\begin{equation*}
\int_{0}^{t}\left\Vert e^{\varepsilon ^{-2}\mathcal{A}_{s}\tau }\psi
(0)\right\Vert _{\mathcal{L}_{n}^{p}}^{q}d\tau \leq
c\int_{0}^{T}e^{-q\varepsilon ^{-2}\omega \tau }\left\Vert \psi
(0)\right\Vert _{\mathcal{L}_{n}^{p}}^{q}d\tau \leq \frac{\varepsilon ^{2}}{%
q\omega }\left\Vert \psi (0)\right\Vert _{\mathcal{L}_{n}^{p}}^{q}.
\end{equation*}
\end{proof}


\section{Averaging over the fast OU-process}


\begin{lemma}
\label{Lemma20}Let Assumption \ref{Wiener} hold and consider $\mathcal{Z}_{i,j}(t)$ 
as defined in (\ref{E40a}). 
Then for arbitrary $\delta
_{0}\in (0,\frac{1}{2})$ we obtain%
\begin{equation}
\mathcal{Z}_{i,j}(t)=\lambda _{j}^{-\frac{1}{2}(1-\delta _{0})}\left( \emph{q%
}_{j,j}^{i,i}\right) ^{\frac{1}{2}}\mathcal{O}(\varepsilon ^{-\delta _{0}}),
\label{z1}
\end{equation}%
and%
\begin{equation}
\mathcal{Z}_{i,j}(t)\mathcal{Z}_{\ell ,k}(t)=\Big(\lambda _{j}\lambda _{k}%
\Big)^{-\frac{1}{2}(1-\delta _{0})}\left( \emph{q}_{j,j}^{i,i}\emph{q}%
_{k,k}^{\ell ,\ell }\right) ^{\frac{1}{2}}\mathcal{O}(\varepsilon ^{-2\delta
_{0}}).  \label{z2}
\end{equation}
Moreover, the\ $\mathcal{O}$-terms are uniform in $i,$ $j,$ $k$ and $\ell .$
\end{lemma}
\begin{proof}
For the first part, we follow the same steps as in Lemma \ref{Lemma2}\ to
obtain%
\begin{equation*}
\mathbb{E}\sup_{t\in [ 0,T_0] }\left\vert \mathcal{Z}_{i,j}(t)\right\vert
^{p}\leq C\varepsilon ^{-2}\left( \lambda _{j}\right) ^{1-\frac{1}{2}%
p}\left( \emph{q}_{j,j}^{i,i}\right) ^{\frac{p}{2}}.
\end{equation*}%
Using H\"{o}lder inequality, we derive for sufficiently large $q$ and for a
constant independent on $i$ and $j$%
\begin{equation*}
\Big(\mathbb{E}\sup_{t\in [ 0,T_0] }\left\vert \mathcal{Z}%
_{i,j}(t)\right\vert ^{p}\Big)^{1/p}\leq C\lambda _{j}^{-\frac{1}{2}%
}(\varepsilon ^{-2}\lambda _{j})^{\frac{1}{pq}}\left( \emph{q}%
_{j,j}^{i,i}\right) ^{\frac{1}{2}}.
\end{equation*}%
We finish the proof by fixing $\delta _{0}=\frac{2}{pq}<\frac{1}{2}$ for
large $q$ and $p.$

For the second part we use Cauchy-Schwarz inequality to obtain%
\begin{equation*}
\mathbb{E}\sup_{[ 0,T_0] }\left\vert \mathcal{Z}_{i,j}\mathcal{Z}_{\ell
,k}\right\vert ^{p}\leq \Big(\mathbb{E}\sup_{[ 0,T_0] }\left\vert \mathcal{Z}%
_{i,j}\right\vert ^{2}\Big)^{1/2}\Big(\mathbb{E}\sup_{\left[ 0,T_{0}\right]
}\left\vert \mathcal{Z}_{\ell ,k}\right\vert ^{2p}\Big)^{1/2}.
\end{equation*}%
Using the first part, yields (\ref{z2}).
\end{proof}
In next corollary we state without proof the general case of Lemma \ref%
{Lemma20}. For the proof we can follow the same steps as in the proof of
Lemma \ref{Lemma20}.
\begin{corollary}
\label{coll2} Under the assumptions of Lemma \ref{Lemma20} we have%
\begin{equation}
\mathop{\textstyle \prod }\limits_{j=1}^{N}\mathcal{Z}_{\ell _{j},k_{j}}=%
\Big(\mathop{\textstyle \prod }\limits_{j=1}^{N}\lambda _{k_{j}}\Big)^{-%
\frac{1}{2}(1-\delta _{0})}\Big( \mathop{\textstyle \prod }\limits_{j=1}^{N}%
\emph{q}_{k_{j},k_{j}}^{\ell _{j},\ell _{j}}\Big) ^{\frac{1}{2}}\mathcal{O}%
(\varepsilon ^{-N\delta _{0}}).  \label{z3}
\end{equation}
\end{corollary}
\begin{lemma}
\label{Lemma21}Let the assumptions of Lemma \ref{Lemma20} hold and let $X$
be a real valued stochastic process such that for some small $r\geq 0$\ we
have $X(0)=\mathcal{O}(\varepsilon ^{-r})$. If $dX=GdT$ with $G=\mathcal{O}%
(\varepsilon ^{-r})$, then%
\begin{equation}
\sup_{t\ge0}\mathbb{E} \vert \mathcal{Z}_{i,j}(t)\vert ^{2}\leq \frac{\emph{q%
}_{j,j}^{i,i}}{2d_{i}\lambda _{j}},  \label{z3a}
\end{equation}%
\begin{equation}
\int_{0}^{t}X\mathcal{Z}_{i,j}d\tilde{W}_{k,m}= \Big( \frac{\emph{q}%
_{m,m}^{k,k}\emph{q}_{j,j}^{i,i}}{\lambda _{j}}\Big) ^{\frac{1}{2}}\mathcal{O%
}(\varepsilon ^{-r}),  \label{z3b}
\end{equation}%
and%
\begin{equation}
\int_{0}^{t}X\mathop{\textstyle \prod }\limits_{\substack{ j=1,  \\ j\neq i}}%
^{N}\mathcal{Z}_{\ell _{j},k_{j}}d\tilde{W}_{\ell _{i},k_{i}}=\Big(%
\mathop{\textstyle \prod }\limits_{\substack{ j=1,  \\ j\neq i}}^{N}\lambda
_{k_{j}}\Big)^{-\frac{1}{2}}\Big(\mathop{\textstyle \prod }\limits_{j=1}^{N}%
\emph{q}_{k_{j},k_{j}}^{\ell _{j},\ell _{j}}\Big)^{\frac{1}{2}}\mathcal{O}%
(\varepsilon ^{-r}).  \label{z3d}
\end{equation}%
Again all\ $\mathcal{O}$-terms are uniform in the indices $\ell_j$ and $k_j$.
\end{lemma}
\begin{proof}
For the first part, we use It\^{o} isometry to obtain%
\begin{equation*}
\mathbb{E}\left\vert \mathcal{Z}_{i,j}\right\vert ^{2}=\frac{1}{\varepsilon
^{2}}\mathbb{E}\Big\vert \int_{0}^{t}e^{-\varepsilon ^{-2}d_{i}(t-s)\lambda
_{j}}d\tilde{W}_{i,j}\Big\vert ^{2}=\frac{\emph{q}_{j,j}^{i.i}}{\varepsilon
^{2}}\int_{0}^{t}e^{-2\varepsilon ^{-2}d_{i}(t-s)\lambda _{j}}ds\leq \frac{%
\emph{q}_{j,j}^{i,i}}{2d_{i}\lambda _{j}}.
\end{equation*}
For the second part using Burkholder-Davis-Gundy theorem and H\"{o}lder
inequality, yields%
\begin{eqnarray*}
\mathbb{E}\sup_{t\in [ 0,T_{0}] }\big|\int_{0}^{t}X\mathcal{Z}_{i,j}d\tilde{W%
}_{k,m}\big|^{p} &\leq &C_{p}\left( \emph{q}_{m,m}^{k,k}\right) ^{\frac{p}{2}%
}\mathbb{E}\Big(\int_{0}^{T_{0}}\left\vert X\right\vert ^{2}\big|\mathcal{Z}%
_{i,j}\big|^{2}d\sigma \Big)^{\frac{p}{2}} \\
&\leq &C_{p,T_{0}}\varepsilon ^{-pr}\left( \emph{q}_{m,m}^{k,k}\right) ^{%
\frac{p}{2}}\int_{0}^{T_{0}}\mathbb{E}|\mathcal{Z}_{i,j}|^{p}d\sigma .
\end{eqnarray*}%
By Gaussianity and the first part we obtain%
\begin{eqnarray*}
\mathbb{E}\sup_{t\in [ 0,T_0] }\big|\int_{0}^{t}X\mathcal{Z}_{i,j}d\tilde{W}%
_{k,m}\big|^{p} &\leq &C_{p,T_{0}}\varepsilon ^{-pr}\Big( \frac{\emph{q}%
_{m,m}^{k,k}\emph{q}_{j,j}^{i,i}}{\lambda _{j}}\Big) ^{\frac{p}{2}}.
\end{eqnarray*}%
Analogously, for the last term%
\begin{eqnarray*}
\mathbb{E}\sup_{t\in [ 0,T_0] }\big|\int_{0}^{t}X\mathop{\textstyle \prod }%
\limits _{\substack{ j=1,  \\ j\neq i}}^{N}\mathcal{Z}_{\ell _{j},k_{j}}d%
\tilde{W}_{\ell _{i},k_{i}}\big|^{p} &\leq &C_{p,T_{0}}\left( \emph{q}%
_{k_{i},k_{i}}^{\ell _{i},\ell _{i}}\right) ^{\frac{p}{2}}\mathbb{E}%
\int_{0}^{T_{0}}\left\vert X\right\vert ^{p}\mathop{\textstyle \prod }\limits
_{\substack{ j=1,  \\ j\neq i}}^{N}\big|\mathcal{Z}_{\ell _{j},k_{j}}\big|%
^{p}d\sigma .
\end{eqnarray*}%
Using H\"older, Gaussianity and the first part,
we obtain (\ref{z3d}).
\end{proof}

In the following we state and prove the averaging lemma over the fast
OU-process $\mathcal{Z}_{i,j}$ (cf. (\ref{E40a})).
\begin{lemma}
\label{Lemma4} Under Assumption \ref{assg}, \ref{Wiener} and \ref{W-C}, let $%
X$ be as in Lemma \ref{Lemma21} and $N\leq m$. Then for $N$ odd%
\begin{equation}
\int_{0}^{t}X\mathop{\textstyle \prod }\limits_{i=1}^{N}\mathcal{Z}_{\ell
_{i},k_{i}}ds=A_{k_{1},\cdots ,k_{N}}^{\ell _{1},\cdots \ell _{N}}\mathcal{O}%
(\varepsilon ^{1-r}),  \label{z7}
\end{equation}%
and for $N$ even%
\begin{eqnarray}
\int_{0}^{t}X\mathop{\textstyle \prod }\limits_{i=1}^{N}\mathcal{Z}_{\ell
_{i},k_{i}}ds &=&\frac{1}{2^{\frac{N}{2}}} \sum\limits_{j\in Per(N)} %
\mathop{\textstyle \prod }\limits_{\eta=1}^{N/2} \frac{\emph{q}%
_{k_{j_{2\eta-1 }},k_{j_{2\eta }} }^{\ell_{j_{2\eta-1 }},\ell _{j_{2\eta}}}}{%
d_{\ell _{j_{2\eta-1 }}} \lambda_{k_{j_{2\eta-1 }}}+d_{\ell
_{j_{2\eta}}}\lambda _{k_{j_{2\eta}}}}\int_{0}^{t}Xds  \notag \\
&&+A_{k_{1},\cdots ,k_{N}}^{\ell _{1},\cdots \ell _{N}}\mathcal{O}%
(\varepsilon ^{1-r})\text{,\ }  \label{z8}
\end{eqnarray}%
with%
\begin{equation*}
\sum\limits_{k_{1}=1}^{\infty }\cdots \sum\limits_{k_{N}=1}^{\infty
}A_{k_{1},\cdots ,k_{N}}^{\ell _{1},\cdots ,\ell _{N}}\mathop{\textstyle
\prod }\limits_{i=1}^N \lambda_{k_i}^{\gamma_1}<\infty .
\end{equation*}
The $\mathcal{O}$-terms are again uniform in all indices.
\end{lemma}
We used $j\in Per(N)$\ if $j=(j_{1},\ldots ,j_{N})$ is a permutation of $%
\left\{ 1,\ldots ,N\right\} .$
\begin{remark}
The term%
\begin{equation*}
\sum\limits_{j\in Per(N)}\mathop{\textstyle \prod }\limits_{\eta=1}^{N/2} 
\frac{\emph{q}_{k_{j_{2\eta-1 }},k_{j_{2\eta }} }^{\ell_{j_{2\eta-1 }},\ell
_{j_{2\eta}}}}{d_{\ell _{j_{2\eta-1 }}} \lambda_{k_{j_{2\eta-1 }}}+d_{\ell
_{j_{2\eta}}}\lambda _{k_{j_{2\eta}}}}
\end{equation*}%
is summable over $k_{1},\cdots ,k_{N}$ by Condition (\ref{e4a}).
\end{remark}
Let us state explicitly some $A$'s appearing in the proof of the
theorem.
\begin{example}
For $N=1$ we have $A_{k}^{\ell }=\frac{1}{\lambda _{k}}\left( \emph{q}%
_{k,k}^{\ell ,\ell }\right) ^{\frac{1}{2}}, $ and for $N=2$%
\begin{equation*}
A_{k_{1},k_{2}}^{\ell _{1},\ell _{2}}=
\Big(\sum\limits_{i=1}^{2}d_{\ell _{i}}\lambda _{k_{i}}\Big)^{-1/2}
\Big(
\mathop{\textstyle \prod }\limits_{i=1}^{2}\lambda _{k_{i}}^{-1}q_{k_{i},k_{i}}^{\ell _{i},\ell _{i}}\Big)^{1/2},
\end{equation*}%
and for $N=3$%
\begin{equation*}
A_{k_{1},k_{2},k_{3}}^{\ell _{1},\ell _{2},\ell _{3}}=\Big(\frac{1}{%
\sum\limits_{i=1}^{3}d_{\ell _{i}}\lambda _{k_{i}}}
\mathop{\textstyle \prod }\limits_{i=1}^{3}\lambda _{k_{i}}^{-1}\emph{q}_{k_{i},k_{i}}^{\ell _{i},\ell
_{i}}\Big)^{\frac{1}{2}}+\sum\limits_{\substack{ j_{1},j_{2}=1  \\ j_{1}\neq
j_{2}}}^{3}\frac{\emph{q}_{k_{j_{1}},k_{j_{2}}}^{\ell _{j_{1}},\ell _{j_{2}}}%
}{d_{\ell _{j_{1}}}\lambda _{k_{j_{1}}}+d_{\ell _{j_{2}}}\lambda _{k_{j_{2}}}%
}\frac{\left( \emph{q}_{j,j}^{i,i}\right) ^{\frac{1}{2}}}{\lambda _{j}}.
\end{equation*}
For larger $N$ the terms have similar structure, but there are about $N/2$
many.
\end{example}
\begin{proof}
Fix a small $\delta _{0}<\frac{1}{N}$ for $N>1.$ First, recall $\vert
X\vert = \mathcal{O}(\varepsilon ^{-r}). $ For the first part we treat $%
N=1$ and $3$. The general case follows by induction.

For $N=1$ we apply It\^{o} formula to the term $X\mathcal{Z}_{i,j}$ to obtain%
\begin{equation*}
\int_{0}^{t}X\mathcal{Z}_{i,j}ds=-\frac{\varepsilon ^{2}}{d_{i}\lambda _{j}}%
X(t)\mathcal{Z}_{i,j}(t)+\frac{\varepsilon ^{2}}{d_{i}\lambda _{j}}%
\int_{0}^{t}G\mathcal{Z}_{i,j}ds+\frac{\varepsilon }{d_{i}\lambda _{j}}%
\int_{0}^{t}Xd\tilde{W}_{i,j}.
\end{equation*}%
Using Lemmas \ref{Lemma20}\ and Burkholder-Davis-Gundy theorem, yields%
\begin{eqnarray}
\int_{0}^{t}X\mathcal{Z}_{i,j}ds &=&\left( \emph{q}_{j,j}^{i,i}\right) ^{%
\frac{1}{2}}\Big[\frac{1}{\left( d_{i}\lambda _{j}\right) \lambda _{j}^{%
\frac{1}{2}-\frac{1}{2}\delta _{0}}}\mathcal{O}(\varepsilon ^{2-r-\delta
_{0}})+\frac{1}{d_{i}\lambda _{j}}\mathcal{O}(\varepsilon ^{1-r})\Big] 
\notag \\
&=&\frac{1}{\lambda _{j}}\left( \emph{q}_{j,j}^{i,i}\right) ^{\frac{1}{2}}%
\mathcal{O}(\varepsilon ^{1-r}).  \label{z5}
\end{eqnarray}%
For $N\in \{3,5,..\}$ we apply It\^{o} formula to the term $X%
\mathop{\textstyle \prod }\limits_{i=1}^{N}\mathcal{Z}_{\ell _{i},k_{i}}$ to
obtain%
\begin{eqnarray*}
\int_{0}^{t}X\mathop{\textstyle \prod }\limits_{i=1}^{N}\mathcal{Z}_{\ell
_{i},k_{i}}ds &=&\frac{1}{\sum\limits_{i=1}^{N}d_{\ell _{i}}\lambda _{k_{i}}}%
\left\{ \varepsilon ^{2}X\mathop{\textstyle \prod }\limits_{i=1}^{N}\mathcal{%
Z}_{\ell _{i},k_{i}}+\varepsilon ^{2}\int_{0}^{t}G\mathop{\textstyle \prod }%
\limits_{i=1}^{N}\mathcal{Z}_{\ell _{i},k_{i}}ds\right. \text{ \ \ \ \ \ \ \
\ \ } \\
&&+\varepsilon \sum\limits_{j=1}^{N}\int_{0}^{t}X\mathop{\textstyle \prod }%
\limits_{\substack{ i=1,  \\ i\neq j}}^{N}\mathcal{Z}_{\ell _{i},k_{i}}d%
\tilde{W}_{\ell _{j},k_{j}} \\
&&+\sum\limits_{j_{1}\neq j_{2}=1}^{N}\int_{0}^{t}X\mathop{\textstyle \prod }%
\limits_{\substack{ i=1,  \\ i\notin \{j_{1,}j_{2}\}}}^{N}\mathcal{Z}_{\ell
_{i},k_{i}}\left. d\tilde{W}_{\ell _{j_{1}},k_{j_{1}}}d\tilde{W}_{\ell
_{j_{2}},k_{j_{2}}}\right\} .
\end{eqnarray*}%
Using Corollary \ref{coll2} and Lemma \ref{Lemma21} to obtain%
\begin{eqnarray*}
\int_{0}^{t}X\mathop{\textstyle \prod }\limits_{i=1}^{N}\mathcal{Z}_{\ell
_{i},k_{i}}ds &=&\frac{\Big(\mathop{\textstyle \prod }\limits_{i=1}^{N}%
\lambda _{k_{i}}^{-1}\emph{q}_{k_{i},k_{i}}^{\ell _{i},\ell _{i}}\Big)^{%
\frac{1}{2}}}{\sum\limits_{i=1}^{N}d_{\ell _{i}}\lambda _{k_{i}}}\left\{ %
\mathop{\textstyle \prod }\limits_{i=1}^{N}\lambda _{k_{i}}^{\frac{1}{2}%
\delta _{0}}\mathcal{O}(\varepsilon ^{2-r-3\delta
_{0}})+\sum\limits_{i=1}^{N}\lambda _{k_{i}}^{\frac{1}{2}}\mathcal{O}%
(\varepsilon ^{1-r})\right\} \\
&&+\frac{1}{\sum\limits_{i=1}^{N}d_{\ell _{i}}\lambda _{k_{i}}}\sum\limits 
_{\substack{ j_{1},j_{2}=1,  \\ j_{1}\neq j_{2}}}^{N}\emph{q}%
_{k_{j_{1}},k_{j_{2}}}^{\ell _{j_{1}},\ell _{j_{2}}}\int_{0}^{t}X%
\mathop{\textstyle \prod }\limits_{\substack{ i=1,  \\ i\notin
\{j_{1,}j_{2}\}}}^{N}\mathcal{Z}_{\ell _{i},k_{i}}ds.
\end{eqnarray*}%
We use $\sum_{i=1}^{N}d_{\ell _{i}}\lambda _{k_{i}}\geq
c\prod_{i=1}^{N}\lambda _{k_{i}}^{1/N}$ with $c=\mathop{\textstyle \prod }%
_{i=1}^{N}d_{\ell _{j_{i}}}^{1/N}$ and the equivalence of norms in $%
\mathbb{R}
^{N}$ which implies for $C_{1},C_{2}>0$%
\begin{equation}
C_{1}\Big(\sum_{i=1}^{N}\lambda _{k_{i}}\Big)^{\frac{1}{2}}\leq
\sum_{i=1}^{N}\lambda _{k_{i}}^{\frac{1}{2}}\leq C_{2}\Big(%
\sum_{i=1}^{N}\lambda _{k_{i}}\Big)^{\frac{1}{2}}.  \label{e100}
\end{equation}%
Hence,%
\begin{eqnarray}
\int_{0}^{t}X\mathop{\textstyle \prod }\limits_{i=1}^{N}\mathcal{Z}_{\ell
_{i},k_{i}}ds &=&\Big(\frac{1}{\sum\limits_{i=1}^{N}d_{\ell _{i}}\lambda
_{k_{i}}}\mathop{\textstyle \prod }\limits_{i=1}^{N}\lambda _{k_{i}}^{-1}%
\emph{q}_{k_{i},k_{i}}^{\ell _{i},\ell _{i}}\Big)^{\frac{1}{2}}\mathcal{O}%
(\varepsilon ^{1-r})  \label{e101} \\
&&+\frac{1}{\sum\limits_{i=1}^{N}d_{\ell _{i}}\lambda _{k_{i}}}\sum\limits 
_{\substack{ j_{1},j_{2}=1  \\ j_{1}\neq j_{2}}}^{N}\emph{q}%
_{k_{j_{1}},k_{j_{2}}}^{\ell _{j_{1}},\ell _{j_{2}}}\int_{0}^{t}X%
\mathop{\textstyle \prod }\limits_{\substack{ i=1,  \\ i\notin
\{j_{1,}j_{2}\}}}^{N}\mathcal{Z}_{\ell _{i},k_{i}}ds.  \notag
\end{eqnarray}%
In the case $N=3,$\ for example, Equation (\ref{e101}) takes the form%
\begin{eqnarray*}
\int_{0}^{t}X\mathop{\textstyle \prod }\limits_{i=1}^{3}\mathcal{Z}_{\ell
_{i},k_{i}}ds &=&\Big(\frac{1}{\sum\limits_{i=1}^{3}d_{\ell _{i}}\lambda
_{k_{i}}}\mathop{\textstyle \prod }\limits_{i=1}^{3}\lambda _{k_{i}}^{-1}%
\emph{q}_{k_{i},k_{i}}^{\ell _{i},\ell _{i}}\Big)^{\frac{1}{2}}\mathcal{O}%
(\varepsilon ^{1-r}) \\
&&+\sum\limits_{\substack{ j_{1},j_{2}=1  \\ j_{1}\neq j_{2}}}^{3}\frac{%
\emph{q}_{k_{j_{1}},k_{j_{2}}}^{\ell _{j_{1}},\ell _{j_{2}}}}{d_{\ell
_{j_{1}}}\lambda _{k_{j_{1}}}+d_{\ell _{j_{2}}}\lambda _{k_{j_{2}}}}\frac{%
\left( \emph{q}_{j,j}^{i,i}\right) ^{\frac{1}{2}}}{\lambda _{j}}\mathcal{O}%
(\varepsilon ^{1-r}),
\end{eqnarray*}%
where we used Equation (\ref{z5}) and $d_{\ell _{j_{1}}}\lambda
_{k_{j_{1}}}+d_{\ell _{j_{2}}}\lambda _{k_{j_{2}}}\leq \sum_{i=1}^{3}d_{\ell
_{i}}\lambda _{k_{i}}$ for $j_{1},j_{2}\in \{1,\ 2,\ 3\}.$ The general case
for $N\in \{5,7,\cdots \}$ follows similarly.

We prove the second part only for $N=2$ and we can proceed by induction. 
Applying
It\^{o} formula to $X\cdot\mathop{\textstyle \prod }_{i=1}^{2}\mathcal{Z}%
_{\ell _{i},k_{i}}$ and integrating from $0$ to $t,$ we obtain 
\begin{align*}
\int_{0}^{t}X\mathop{\textstyle \prod }\limits_{i=1}^{2}\mathcal{Z}_{\ell
_{i},k_{i}}ds &=\frac{1}{\sum\limits_{i=1}^{2}d_{\ell _{i}}\lambda _{k_{i}}}%
\left\{ -\varepsilon ^{2}X(t)\mathop{\textstyle \prod }\limits_{i=1}^{2}%
\mathcal{Z}_{\ell _{i},k_{i}}(t)+\varepsilon ^{2} \int_{0}^{t}G%
\mathop{\textstyle \prod }\limits_{i=1}^{2}\mathcal{Z}_{\ell
_{i},k_{i}}ds\right. \\
&+\varepsilon \sum\limits_{j=1}^{2}\int_{0}^{t}X\mathop{\textstyle \prod }%
\limits_{\substack{ i=1 \\ i\neq j}}^{2}\mathcal{Z}_{\ell _{i},k_{i}}d\tilde{%
W}_{\ell _{j},k_{j}}+\int_{0}^{t}X\left. d\tilde{W}_{\ell
_{j_{1}},k_{j_{1}}}d\tilde{W}_{\ell _{j_{2}},k_{j_{2}}}\right\} .
\end{align*}%
Using Corollary \ref{coll2} and Lemma \ref{Lemma21} to obtain%
\begin{eqnarray*}
\int_{0}^{t}X\mathop{\textstyle \prod }\limits_{i=1}^{2}\mathcal{Z}_{\ell
_{i},k_{i}}ds &=&\frac{1}{\sum\limits_{i=1}^{2}d_{\ell _{i}}\lambda _{k_{i}}}%
\left( \mathop{\textstyle \prod }\limits_{i=1}^{2}\lambda _{k_{i}}^{-1}\emph{%
q}_{k_{i},k_{i}}^{\ell _{i},\ell _{i}}\right) ^{\frac{1}{2}}\left\{ \mathcal{%
O}(\varepsilon ^{1-r-2\delta _{0}})\mathop{\textstyle \prod }%
\limits_{i=1}^{2}\lambda _{k_{i}}^{\frac{1}{2}\delta _{0}}\right. \\
&&+\left. \mathcal{O}(\varepsilon ^{1-r})\sum\limits_{i=1}^{2}\lambda
_{k_{i}}^{\frac{1}{2}}\right\} +\frac{\emph{q}_{k_{1},k_{2}}^{\ell _{1},\ell
_{2}}}{\sum\limits_{i=1}^{2}d_{\ell _{i}}\lambda _{k_{i}}}\int_{0}^{t}Xds.
\end{eqnarray*}
Using (\ref{e100}) with $\Big( \sum_{i=1}^{2}d_{\ell _{i}}\lambda _{k_{i}}%
\Big) ^{\frac{1}{2}}\geq c\mathop{\textstyle \prod }_{i=1}^{2}\lambda
_{k_{i}}^{\frac{1}{4}}$\ we obtain for $\delta _{0}<\frac{1}{2}$ 
\begin{equation*}
\int_{0}^{t}X\mathop{\textstyle \prod }\limits_{i=1}^{2}\mathcal{Z}_{\ell
_{i},k_{i}}ds=\Big(\frac{1}{\sum\limits_{i=1}^{2}d_{\ell _{i}}\lambda
_{k_{i}}}\mathop{\textstyle \prod }\limits_{i=1}^{2}\lambda _{k_{i}}^{-1}%
\emph{q}_{k_{i},k_{i}}^{\ell _{i},\ell _{i}}\Big)^{\frac{1}{2}}\mathcal{O}%
(\varepsilon ^{1-r})+\frac{\emph{q}_{k_{1},k_{2}}^{\ell _{1},\ell _{2}}}{%
\sum\limits_{i=1}^{2}d_{\ell _{i}}\lambda _{k_{i}}}\int_{0}^{t}Xds.
\end{equation*}
For $N\in \{4,6,..\}$ we apply It\^{o} formula to the term $X%
\mathop{\textstyle \prod }\limits_{i=1}^{N}\mathcal{Z}_{\ell _{i},k_{i}}$ to
obtain%
\begin{eqnarray*}
\int_{0}^{t}X\mathop{\textstyle \prod }\limits_{i=1}^{N}\mathcal{Z}_{\ell
_{i},k_{i}}ds &=&\frac{1}{\sum\limits_{i=1}^{N}d_{\ell _{i}}\lambda _{k_{i}}}%
\left\{ -\varepsilon ^{2}X(t)\mathop{\textstyle \prod }\limits_{i=1}^{N}%
\mathcal{Z}_{\ell _{i},k_{i}}(t)+\varepsilon ^{2}\int_{0}^{t}G%
\mathop{\textstyle \prod }\limits_{i=1}^{N}\mathcal{Z}_{\ell
_{i},k_{i}}ds\right. \\
&& \qquad +\varepsilon \sum\limits_{j=1}^{N}\int_{0}^{t}X\mathop{\textstyle
\prod }\limits_{\substack{ i=1,  \\ i\neq j}}^{N}\mathcal{Z}_{\ell
_{i},k_{i}}d\tilde{W}_{\ell _{j},k_{j}} \\
&&\qquad +\sum\limits_{j_{1}\neq j_{2}=1}^{N}\int_{0}^{t}X\mathop{\textstyle
\prod }\limits_{\substack{ i=1,  \\ i\neq j_{1}\neq j_{2}}}^{N}\mathcal{Z}%
_{\ell _{i},k_{i}}\left. d\tilde{W}_{\ell _{j_{1}},k_{j_{1}}}d\tilde{W}%
_{\ell _{j_{2}},k_{j_{2}}}\right\} .
\end{eqnarray*}%
Using Corollary \ref{coll2} and Lemma \ref{Lemma21} to obtain as in the odd
case before%
\begin{eqnarray*}
\int_{0}^{t}X\mathop{\textstyle \prod }\limits_{i=1}^{N}\mathcal{Z}_{\ell
_{i},k_{i}}ds &=&\Big(\frac{1}{\sum\limits_{i=1}^{N}d_{\ell _{i}}\lambda
_{k_{i}}}\mathop{\textstyle \prod }\limits_{i=1}^{N}\lambda _{k_{i}}^{-1}%
\emph{q}_{k_{i},k_{i}}^{\ell _{i},\ell _{i}}\Big)^{\frac{1}{2}}\mathcal{O}%
(\varepsilon ^{1-r}) \\
&&+\sum\limits _{\substack{ j_{1},j_{2}=1  \\ j_{1}\neq j_{2}}}^{N} \frac{%
\emph{q}_{k_{j_{1}},k_{j_{2}}}^{\ell _{j_{1}},\ell _{j_{2}}} }{
\sum\limits_{i=1}^{N}d_{\ell _{i}}\lambda _{k_{i}}}\cdot \int_{0}^{t}X \!\!%
\mathop{\textstyle \prod }\limits_{\substack{ i=1,  \\ i\notin
\left\{j_{1},j_{2}\right\} }}^{N}\!\! \mathcal{Z}_{\ell _{i},k_{i}}ds.
\end{eqnarray*}%
The first factor in the sum is summable over $j_1$ and $j_2$ by Condition (%
\ref{e4a}). Now, we can proceed by induction and apply the assertion for $N-2
$ to obtain (\ref{z8}).
\end{proof}

\begin{lemma}
\label{Lemmach3.4} Under Assumption \ref{assg}, \ref{Wiener} and \ref{W-C}
let $X$ be as in Lemma \ref{Lemma4}. Then, for $\ell \in \mathbb{N}%
_{0}^{n}$ with $m \geq\vert \ell \vert \geq 1$, we obtain:

1- If one of the $\ell _{i}$ is odd,\ then%
\begin{equation}
P_{c}\int_{0}^{t}X( \mathcal{Z}^{s}) ^{\ell }d\tau =\mathcal{O}(\varepsilon
^{1-r}).  \label{e20}
\end{equation}

2-If all $\ell _{i}$\ are even, then there is a constant $C_{\ell }$\ such
that%
\begin{equation}
P_{c}\int_{0}^{t}X( \mathcal{Z}^{s}) ^{\ell }d\tau =C_{\ell
}\int_{0}^{t}Xd\tau +\mathcal{O}(\varepsilon ^{1-r}),  \label{e21}
\end{equation}%
where $C_{\ell }$ is given by%
\begin{equation}
C_{\ell }=\mathop{\textstyle \prod }\limits_{i=1}^{n} \Big(\frac{1}{2^{\ell
_{i}/2}d_{i}^{\ell_{i}/2}} \sum_{k_{1},.,k_{\ell _{i}}=1}^{\infty }
\sum\limits_{j\in Per(\ell _{i})} \mathop{\textstyle \prod }%
\limits_{\eta=1}^{\ell_i/2} \frac{\emph{q}_{k_{j_{2\eta-1
}},k_{j_{2\eta}}}^{i,i}}{ \lambda _{k_{j_{2\eta-1}}}+\lambda _{k_{j_{2\eta}}}%
} P_{c}\mathop{\textstyle \prod }\limits_{\eta=1}^{\ell_{i}}g_{k_{\eta}}\Big)%
.  \label{z10}
\end{equation}
\end{lemma}

\begin{proof}
From the definition of $\mathcal{Z}^{s}$ (cf. (\ref{E40}), we obtain%
\begin{equation}
( \mathcal{Z}^{s}) ^{\ell }=\mathop{\textstyle \prod }\limits_{i=1}^{n}%
\mathcal{Z}_{i}^{\ell _{i}}=\mathop{\textstyle \prod }\limits_{i=1}^{n}\Big(%
\sum\limits_{j_{1},\cdots ,j_{\ell _{i}}=1}^{\infty }\mathop{\textstyle
\prod }\limits_{k=1}^{\ell _{i}}\mathcal{Z}_{i,j_{k}}g_{j_{k}}\Big).
\label{z10d}
\end{equation}
We focus in the proof on the case $n=1$ and $n=2$ as they are needed for our
applications. The general case follows similarly but it is technically more
involved. For $n=1$ we have $\ell =\ell _{1}$ and 
\begin{equation*}
P_{c}\int_{0}^{t}X( \mathcal{Z}^{s}) ^{\ell }d\tau
=P_{c}\sum\limits_{j_{1},\cdots ,j_{\ell }=1}^{\infty }\mathop{\textstyle
\prod }\limits_{k=1}^{\ell }g_{j_{k}}\int_{0}^{t}X\mathop{\textstyle \prod }%
\limits_{k=1}^{\ell }\mathcal{Z}_{1,j_{k}}d\tau .
\end{equation*}%
Now we consider two cases. First if $|\ell| $ is odd, then Lemma \ref{Lemma4}
with $N=|\ell| $ yields 
\begin{equation*}
\int_{0}^{t}X( \mathcal{Z}^{s}) ^{\ell }d\tau =\sum\limits_{k_{1}=1}^{\infty
}\cdots \sum\limits_{k_{\ell }=1}^{\infty }A_{k_{1},\cdots ,k_{\ell
}}^{1,\cdots ,1}\mathop{\textstyle \prod }\limits_{j=1}^{\ell }g_{k_j}\cdot 
\mathcal{O}(\varepsilon ^{1-r}).
\end{equation*}%
And then as the $A$'s are summable%
\begin{equation*}
P_{c}\int_{0}^{t}X( \mathcal{Z}^{s}) ^{\ell }d\tau =\mathcal{O}(\varepsilon
^{1-r}).
\end{equation*}%
Secondly, if $|\ell| $ is even, then Lemma \ref{Lemma4} implies 
\begin{eqnarray*}
\int_{0}^{t}X( \mathcal{Z}^{s}) ^{\ell }d\tau &=&\sum_{k_{1},\cdots ,k_{\ell
}=1}^{\infty } \frac{1}{2^{\frac{\ell }{2}}d_{1}^{\frac{\ell }{2}}}
\sum_{j\in Per(\ell )} \mathop{\textstyle \prod }\limits_{\eta=1}^{|\ell|/2} 
\frac{\emph{q}_{k_{j_{2\eta-1 }},k_{j_{2\eta}}}^{1,1}}{ \lambda
_{k_{j_{2\eta-1 }}}+\lambda _{k_{j_{2\eta}}}}\mathop{\textstyle \prod }%
\limits_{\eta=1}^{\ell }g_{k_\eta} \int_{0}^{t}X ds \\
&&+\sum\limits_{k_{1}=1}^{\infty }\cdots \sum_{k_{\ell
}=1}^{\infty}A_{k_{1},\cdots ,k_{\ell }}^{1,\cdots ,1} \mathop{\textstyle
\prod }\limits_{\eta=1}^{\ell }g_{k_\eta} \mathcal{O}(\varepsilon ^{1-r}).
\end{eqnarray*}%
As the $A$'s are summable%
\begin{eqnarray*}
P_{c}\!\int_{0}^{t}\!X( \mathcal{Z}^{s}) ^{\ell }d\tau\!\!
&=&\!\!\!\!\sum\limits_{k_{1},\cdots ,k_{\ell }=1}^{\infty }\frac{1}{2^{%
\frac{\ell }{2}}d_{1}^{\frac{\ell }{2}}}\sum\limits_{j\in Per(\ell )} %
\mathop{\textstyle \prod }\limits_{\eta=1}^{|\ell|/2} \frac{\emph{q}%
_{k_{j_{2\eta-1 }},k_{j_{2\eta}}}^{1,1}}{ \lambda
_{k_{j_{2\eta-1}}}\!+\lambda _{k_{j_{2\eta}}}}P_{c}\mathop{\textstyle \prod }%
\limits_{\eta=1}^{\ell }g_{k_\eta} \int_{0}^{t}Xds \\
&& + \mathcal{O}(\varepsilon ^{1-r}).
\end{eqnarray*}%
For $n=2,$ we have $N=|\ell| =\ell _{1}+\ell _{2}$ and from Equation (\ref%
{z10d}) 
\begin{equation*}
\int_{0}^{t}X( \mathcal{Z}^{s}) ^{\ell }d\tau =\sum\limits_{j_{1}=1}^{\infty
} \cdots \sum\limits_{j_{|\ell|}=1}^{\infty } \int_{0}^{t}X%
\mathop{\textstyle \prod }\limits_{k=1}^{|\ell| }\mathcal{Z}%
_{i_{k},j_{k}}d\tau \mathop{\textstyle \prod }\limits_{k=1}^{|%
\ell|}g_{j_{k}},
\end{equation*}%
with $i_{1}=\cdots =i_{\ell _{1}}=1$ and $i_{\ell _{1}+1}=\cdots
=i_{|\ell|}=2.$ Similarly to the first part, we consider two cases. First if 
$|\ell|$ is odd, then we apply Lemma \ref{Lemma4} to obtain%
\begin{equation*}
\int_{0}^{t}X( \mathcal{Z}^{s}) ^{\ell }d\tau =\sum\limits_{j_{1}=1}^{\infty
}\cdots \sum\limits_{j_{|\ell|}=1}^{\infty } A_{k_{j_{1}},\cdots
,k_{j_{|\ell| }}}^{i_{1},\cdots ,i_{|\ell| }}\mathop{\textstyle \prod }%
\limits_{k=1}^{|\ell| }g_{j_{k}}\mathcal{O}(\varepsilon ^{1-r}).
\end{equation*}%
As the $A$'s are summable%
\begin{equation*}
P_{c}\int_{0}^{t}X( \mathcal{Z}^{s}) ^{\ell }d\tau =\mathcal{O}(\varepsilon
^{1-r}).
\end{equation*}%
In the second case, when $|\ell| $ is even, we apply Lemma \ref{Lemma4} and
analogously to the first case we obtain%
\begin{eqnarray*}
\int_{0}^{t}X( \mathcal{Z}^{s}) ^{\ell }d\tau \!\!\! &=&\!\!\!\!\!\!\!\!
\sum_{ j_{1},\ldots,j_{|\ell|}=1}^{\infty } \frac{1}{2^{\frac{|\ell| }{2}}}
\sum_{j\in Per(|\ell|)} \mathop{\textstyle \prod }\limits_{\eta=1
}^{|\ell|/2} \frac{\emph{q}_{j_{k_{2\eta-1}},j_{k_{2\eta}}}^{i_{k_{2\eta-1
}},i_{k_{2\eta}}}}{ d_{i_{k_{2\eta-1}}}\!\lambda _{j_{k_{2\eta-1
}}}\!\!+d_{i_{k_{2\eta}}}\!\lambda _{j_{k_{2\eta}}}}\mathop{\textstyle \prod
}\limits_{k=1}^{\vert \ell\vert }g_{j_{k}}\int_{0}^{t}Xds \\
&&+\sum\limits_{j_{1}=1}^{\infty }\cdots \sum\limits_{j_{|\ell|}=1}^{\infty
} A_{k_{j_{1}},\cdots ,k_{j_{|\ell| }}}^{i_{1},\cdots ,i_{|\ell| }}%
\mathop{\textstyle \prod }\limits_{k=1}^{|\ell| }g_{j_{k}}\mathcal{O}%
(\varepsilon ^{1-r}).
\end{eqnarray*}%
We obtain%
\begin{eqnarray*}
P_{c}\int_{0}^{t}X( \mathcal{Z}^{s}) ^{\ell }d\tau &=&\!\!\!\!\!\!\!\!
\sum_{ j_{1},\ldots,j_{|\ell|}=1}^{\infty } \frac{1}{2^{\frac{|\ell| }{2}}}%
\sum_{j\in Per(|\ell|)} \mathop{\textstyle \prod }\limits_{\eta=1
}^{|\ell|/2} \frac{\emph{q}_{j_{k_{2\eta-1}},j_{k_{2\eta}}}^{i_{k_{2\eta-1
}},i_{k_{2\eta}}}}{ d_{i_{k_{2\eta-1}}}\!\lambda _{j_{k_{2\eta-1
}}}\!\!+d_{i_{k_{2\eta}}}\!\lambda _{j_{k_{2\eta}}}} \\
&&\times P_{c}\Big(\mathop{\textstyle \prod }\limits_{k=1}^{|\ell| }g_{j_{k}}%
\Big)\int_{0}^{t}Xds+\mathcal{O}(\varepsilon ^{1-r}).
\end{eqnarray*}%
We can distinguish between two cases when $|\ell| $ is even. First one of $%
\ell _{1}$ and $\ell _{2}$\ is odd. Here $\emph{q}_{j_{\ell _{1}},j_{\ell
_{1}+1}}^{i_{\ell _{1}},i_{\ell _{1}+1}}=0,$ where $i_{\ell _{1}}=1$ and $%
i_{\ell _{1}+1}=2.$ Thus 
\begin{equation*}
P_{c}\int_{0}^{t}X( \mathcal{Z}^{s}) ^{\ell }d\tau =\mathcal{O}(\varepsilon
^{1-r}).
\end{equation*}%
In the second case when $\ell _{1}$ and $\ell _{2}$\ are both even, we have%
\begin{eqnarray*}
\lefteqn{P_{c}\int_{0}^{t}X( \mathcal{Z}^{s}) ^{\ell }d\tau } \\
&=&P_{c}\!\!\! \times \sum\limits_{j_{1}=1}^{\infty } \ldots
\sum\limits_{j_{\ell _{1}}=1}^{\infty } \frac{1}{(2d_{1})^{\frac{\ell _{1}}{2%
}}}\sum\limits_{j\in Per(\ell_{1})} \mathop{\textstyle \prod }\limits_{\eta
= 1}^{\ell _{1}/2} \frac{{q}_{j_{k_{2\eta-1 }},j_{k_{2\eta}}}^{1,1}}{
\lambda _{j_{k_{2\eta-1}}}+\lambda _{j_{k_{2\eta}}}} \mathop{\textstyle
\prod }\limits_{k=1}^{\ell _{1}}g_{j_{k}} \\
&& \times \sum_{j_{1}=1}^{\infty } \ldots \sum_{j_{\ell _{2}}=1}^{\infty } 
\frac{1}{(2d_{2})^{\frac{\ell _{2}}{2}}} \sum_{j\in Per(\ell_{2})} %
\mathop{\textstyle \prod }\limits_{\eta= 1}^{\ell _{2}/2} \frac{%
q_{j_{k_{2\eta-1 }},j_{k_{2\eta}}}^{2,2}}{ \lambda
_{j_{k_{2\eta-1}}}+\lambda _{j_{k_{2\eta}}}} \mathop{\textstyle \prod }%
\limits_{k=1}^{\ell_{2}}g_{j_{k}} \int_{0}^{t}Xds \\
&& \quad +\ \mathcal{O}(\varepsilon ^{1-r}) \\
&=&\mathop{\textstyle \prod }_{i=1}^{2} \sum_{j_{1}=1}^{\infty } \ldots
\sum_{j_{\ell _{i}}=1}^{\infty } \frac{1}{(2d_i)^{\frac{\ell _{i}}{2}}}%
\sum\limits_{j\in Per(\ell _{1})} \mathop{\textstyle \prod }\limits_{\eta=
1}^{\ell _{i}/2} \frac{q_{j_{k_{2\eta-1 }},j_{k_{2\eta}}}^{2,2}}{ \lambda
_{j_{k_{2\eta-1}}}+\lambda _{j_{k_{2\eta}}}} \int_{0}^{t}Xds \cdot P_{c}%
\mathop{\textstyle \prod }\limits_{k=1}^{|\ell|}g_{j_{k}} \\
&& \quad +\ \mathcal{O}(\varepsilon ^{1-r}).
\end{eqnarray*}%
The general case for $n>2$ follows in a similar way, as the random variables 
$\Big(\sum\limits_{j_{1},\cdots ,j_{\ell _{i}}=1}^{\infty }%
\mathop{\textstyle \prod }\limits_{k=1}^{\ell _{i}}\mathcal{Z}%
_{i,j_{k}}g_{j_{k}}\Big)_{i=1,2,...,n}$\ are independent, and we can thus
glue together the individual averaging results as above.
\end{proof}


\section{Proof of the Approximation Theorem I}

\begin{lemma}
\label{Lemma10}Let Assumptions \ref{Wiener}, \ref{assg} and \ref{Poly} hold.
Then%
\begin{equation}
a_{i}(t)=a_{i}(0)+\int_{0}^{t}\mathcal{F}_{i}(a)d\tau +\sum_{\left\vert \ell
\right\vert =2,4,..}\frac{C_{\ell }}{\ell !}\int_{0}^{t}D^{\ell }\mathcal{F}%
_{i}(a)d\tau +\tilde{R}(t),  \label{Ampli}
\end{equation}%
where $C_{\ell }$ was defined in (\ref{z10}) and the error  is
bounded by 
$
\tilde{R}=\mathcal{O}(\varepsilon ^{1-2m\kappa -\kappa _{0}}).  
$
\end{lemma}
\begin{proof}
The mild formulation of (\ref{E10}) and Lemma \ref{Lemma1} with $\rho=0 
$ yields 
\begin{equation}
\psi (t)=\mathcal{Z}^{s}(t)+e^{\varepsilon ^{-2}t\mathcal{A}}\psi (0)+%
\mathcal{O}(\varepsilon ^{2-m\kappa })=:\mathcal{Z}^{s}(t)+y(t)+R(t),
\label{E100}
\end{equation}
where 
\begin{equation*}
y(t)=e^{\varepsilon ^{-2}t\mathcal{A}}\psi (0)\text{ and }R(t)=\mathcal{O}%
(\varepsilon ^{2-m\kappa }).
\end{equation*}
Substituting from (\ref{E100}) into (\ref{E10a}), yields 
\begin{equation}
a_{i}(t)=a_{i}(0)+\int_{0}^{t}\mathcal{F}_{i}(a+\mathcal{Z}^{s}+y+R)(\tau
)d\tau .  \label{E100a}
\end{equation}
Taylor's expansion for the polynomial $\mathcal{F}_{i}:\mathcal{L}%
_{n}^{p}\rightarrow \mathbb{R}$ yields 
\begin{equation}
a_{i}(t)=a_{i}(0)+\int_{0}^{t}\mathcal{F}_{i}(a+\mathcal{Z}^{s})(\tau )d\tau
+R^{(1)}_i(t),  \label{E10b}
\end{equation}
where $R^{(1)}(t)$ is given by 
\begin{equation*}
R^{(1)}_i(t)=\sum_{|\ell| \geq 1}P_{c}\int_{0}^{t}\frac{D^{\ell }\mathcal{F}%
_{i}(a+\mathcal{Z}^{s})}{\ell !}(y+R)^{\ell }d\tau .
\end{equation*}%
We see later that $R^{(1)}$\ is small, as all terms contain at least
one $R=\mathcal{O}(\varepsilon ^{2-m\kappa })$.  Taylor's expansion
again for the polynomial  $\mathcal{F}_{i}:\mathcal{L}_{n}^{p}\rightarrow 
\mathbb{R}$, yields 
\begin{equation*}
a_{i}(t)=a_{i}(0)+\sum_{|\ell| \geq 0}P_{c}\int_{0}^{t}\frac{D^{\ell }%
\mathcal{F}_{i}(a)}{\ell !}( \mathcal{Z}^{s}) ^{\ell }d\tau +R_{1}(t).
\label{E16}
\end{equation*}
Applying the Averaging-Lemma \ref{Lemmach3.4}, yields 
\begin{equation*}
a_{i}(t)=a_{i}(0)+\sum_{|\ell| \geq 0}\frac{C_{\ell }}{\ell !}%
\int_{0}^{t}D^{\ell }\mathcal{F}_{i}(a)d\tau +\mathcal{O}(\varepsilon
^{1-m_{i}\kappa })+R^{(1)}_i(t),  \label{AM}
\end{equation*}
where $C_{0}=1$ and $C_{\ell }=0$ if one $\ell _{i}$ is odd. Thus 
\begin{equation*}
a_{i}(t)=a_{i}(0)+\sum_{|\ell| =0,2,4,..}\frac{C_{\ell }}{\ell !}%
\int_{0}^{t}D^{\ell }\mathcal{F}_{i}(a)d\tau +\tilde{R}_i(t),
\end{equation*}
where 
$
\tilde{R}(t)=R^{(1)}(t)+\mathcal{O}(\varepsilon ^{1-m\kappa }).
$

To bound $\tilde{R}$ we use Lemmas \ref{Lemma3} and \ref{Lemma2} and
Assumption \ref{Poly}.
\end{proof}


\begin{definition}
Define the set $\Omega ^{\ast }\subset \Omega $ such that all the following
estimates hold on $\Omega ^{\ast }$%
\begin{equation}
\sup_{\lbrack 0,\tau ^{\ast }]}\Vert \psi -\mathcal{Q}\Vert _{\mathcal{L}%
_{n}^{p}}<C\varepsilon ^{2-m\kappa -\kappa }\text{ },  \label{eq30}
\end{equation}%
\begin{equation}
\sup_{\lbrack 0,\tau ^{\ast }]}\| \psi \| _{\mathcal{L}_{n}^{p}}<C%
\varepsilon ^{-\frac{3}{2}\kappa _{0}}\text{ },  \label{eq31}
\end{equation}%
\begin{equation}
\sup_{\lbrack 0,\tau ^{\ast }]}\Vert \tilde{R}\Vert _{\mathcal{L}%
_{n}^{p}}<C\varepsilon ^{1-2m\kappa -\kappa }\text{ ,}  \label{eq32}
\end{equation}%
and%
\begin{equation}
\sup_{\lbrack 0,T_{1}]}|b| \leq \tilde{C}_{0}\text{.}  \label{eq33}
\end{equation}
\end{definition}

%
%

\begin{proposition}
$\Omega ^{\ast }$ has approximately probability $1.$
\end{proposition}

\begin{proof}
\begin{equation*}
\mathbb{P}(\Omega ^{\ast })\geq 1-\mathbb{P}(\sup_{[0,\tau ^{\ast }]}\Vert
\psi -\mathcal{Q}\Vert _{\mathcal{L}_{n}^{p}}\geq C\varepsilon ^{2-m\kappa
-\kappa })-\mathbb{P}(\sup_{[0,\tau ^{\ast }]}\| \psi \| _{\mathcal{L}%
_{n}^{p}}\geq C\varepsilon ^{-\frac{3}{2}\kappa _{0}})
\end{equation*}%
\begin{equation*}
-\mathbb{P}(\sup_{[0,\tau ^{\ast }]}\Vert \tilde{R}\Vert _{\mathcal{L}%
_{n}^{p}}\geq C\varepsilon ^{1-2m\kappa -\kappa })-\mathbb{P}%
(\sup_{[0,T_{1}]}|b| >\tilde{C}_{0}).
\end{equation*}%
Using Chebychev inequality and Lemmas \ref{Lemma1}, \ref{Lemma10} and
Corollary \ref{Cor1}, we obtain for $\kappa >\kappa _{0}$ and sufficiently
large $q>\frac{2p}{(\kappa -\kappa _{0})}>0$ 
\begin{eqnarray}
\mathbb{P}(\Omega ^{\ast }) &\geq &1-C[\varepsilon ^{q\kappa }+\varepsilon ^{%
\frac{1}{2}q\kappa }+\varepsilon ^{q(\kappa -\kappa _{0})}]-\mathbb{P}%
(\sup_{[0,T_{1}]}|b| >\tilde{C}_{0})  \notag \\
&\geq &1-C\varepsilon ^{\frac{1}{2}q(\kappa -\kappa _{0})}-\mathbb{P}%
(\sup_{[0,T_{1}]}|b| >\tilde{C}_{0})  \notag \\
&\geq &1-C\varepsilon ^{p},  \label{e34}
\end{eqnarray}%
where $\tilde{C}_{0}$\ is chosen sufficiently large ($\sup_{[0,T_{1}]}|b|
\leq C$ by Assumption \ref{Amp}).
\end{proof}

\begin{theorem}
\label{th1} Assume that Assumptions \ref{Poly} and \ref{Amp}\ hold. Suppose $%
a(0)=\mathcal{O}(1)$ and $\psi (0)=\mathcal{O}(1)$. Let $b$ be a solution of
(\ref{ODE})\ and $a$ as defined in (\ref{Ampli}). If the initial conditions
satisfy $a(0)=b(0)$, then for $\kappa <\frac{1}{2m+1}$ we obtain%
\begin{equation}
\sup_{t\in [ 0,T_{1}\wedge \tau ^{\ast }] } |a(t)-b(t)| \leq C\varepsilon
^{1-2m\kappa -\kappa } \text{\ \ \ on }\Omega ^{\ast },  \label{e33}
\end{equation}%
and%
\begin{equation}
\sup_{t\in [ 0,T_{1}\wedge \tau ^{\ast }] }\left\vert a(t)\right\vert \leq C%
\text{\ \ \ on }\Omega ^{\ast }.  \label{e33a}
\end{equation}
\end{theorem}
We note that all norms in a finite dimensional space are equivalent. Thus
for simplicity of notation we always use the standard
Euclidean norm.

\begin{proof}
Subtracting (\ref{ODE})\ from (\ref{Ampli}) and defining 
\begin{equation}
h:=a-b,  \label{e36}
\end{equation}%
we obtain%
\begin{equation}
h(t)=\sum_{|\ell| =0,2,4,..}\frac{C_{\ell }}{\ell !}\int_{0}^{t}[D^{\ell }%
\mathcal{F}_{i}(h+b)-D^{\ell }\mathcal{F}_{i}(b)]d\tau +\tilde{R}(t),
\label{e37}
\end{equation}%
where the error $\tilde{R}$ is bounded by 
$
\tilde{R}=\mathcal{O}(\varepsilon ^{1-2m\kappa }). 
$

Define $Q$ as%
\begin{equation}
Q:=h-\tilde{R}.  \label{e35}
\end{equation}%
From Equation (\ref{e37}) we obtain%
\begin{equation*}
\partial _{t}Q=\sum_{|\ell| =0,2,4,..}\frac{C_{\ell }}{\ell !}D^{\ell }[%
\mathcal{F}_{i}(Q+\tilde{R}+b)-D^{\ell }\mathcal{F}_{i}(b)].
\end{equation*}%
Taking the scalar product $\left\langle Q,\cdot \right\rangle $\ on both
sides, yields%
\begin{equation*}
\frac{1}{2}\partial _{t}| Q| ^{2}=\sum_{\left\vert \ell \right\vert
=0,2,4,..}\frac{C_{\ell }}{\ell !}\left\langle D^{\ell }\mathcal{F}_{i}(Q+%
\tilde{R}+b)-D^{\ell }\mathcal{F}_{i}(b),Q\right\rangle .
\end{equation*}%
Using Young and Cauchy-Schwarz inequalities, where $\mathcal{F}$ is a
polynomial of degree $m$, we obtain%
\begin{equation}
\frac{1}{2}\partial _{t}| Q| ^{2}\leq C\left( 1+| Q| ^{m-1}+|\tilde{R}%
|^{m-1}+|b| ^{m-1}\right) \left( | Q| ^{2}+|\tilde{R}|^{2}\right) .
\label{e35a}
\end{equation}%
As long as $| Q| <1,$ using Equations (\ref{eq32}) and (\ref{eq33}), we
obtain for $\kappa <\frac{1}{2m+1}$ 
\begin{equation*}
\frac{1}{2}\partial _{t}| Q| ^{2}\leq c\left\vert Q\right\vert
^{2}+C\varepsilon ^{2-2(2m+1)\kappa }\text{ \ on \ }\Omega ^{\ast },
\end{equation*}%
Using Gronwall's lemma, we obtain for $t\leq\tau ^{\ast }\wedge T_{1}\leq
T_{0}$%
\begin{equation*}
\left\vert Q(t)\right\vert ^{2}\leq C\varepsilon ^{1-(2m+1)\kappa
}e^{2cT_{0}},
\end{equation*}%
and thus $\left\vert Q(t)\right\vert <1$ for $t\leq \tau ^{\ast }\wedge
T_{1} $. Taking supremum on $[0,\tau ^{\ast }\wedge T_{1}]$%
\begin{equation*}
\sup_{t\in \lbrack 0,\tau ^{\ast }\wedge T_{1}]}\left\vert Q(t)\right\vert
^{2}\leq C\varepsilon ^{1-(2m+1)\kappa }\text{ \ on \ }\Omega ^{\ast }.
\end{equation*}%
Hence, 
\begin{eqnarray}
\sup_{\lbrack 0,\tau ^{\ast }\wedge T_{1}]}\left\vert a-b\right\vert
&=&\sup_{[0,\tau ^{\ast }\wedge T_{1}]}|Q-\tilde{R}|\leq \sup_{\lbrack
0,\tau ^{\ast }\wedge T_{1}]}\left\vert Q(t)\right\vert +\sup_{[0,\tau
^{\ast }\wedge T_{1}]}|\tilde{R}|  \notag \\
&\leq &C\varepsilon ^{1-(2m+1)\kappa }\text{ \ on \ }\Omega ^{\ast }.
\label{E37}
\end{eqnarray}%
We finish the proof by using (\ref{e36}), (\ref{e35}) and%
\begin{equation*}
\sup_{\lbrack 0,\tau ^{\ast }\wedge T_{1}]}\left\vert a\right\vert \leq
\sup_{\lbrack 0,\tau ^{\ast }\wedge T_{1}]}\left\vert a-b\right\vert
+\sup_{[0,\tau ^{\ast }\wedge T_{1}]}|b| \leq C.
\end{equation*}
\end{proof}

Now we can collect the results obtained previously to prove the main result
of Theorem \ref{thm} and Corollary \ref{coll5} for the system of SPDE (\ref%
{eq1a}).

\begin{proof}[Proof of Theorem \protect\ref{thm}]
Using (\ref{E7}) and triangle inequality, we obtain%
\begin{eqnarray*}
\sup_{t\in \lbrack 0,T_{1}\wedge \tau ^{\ast }]}\Vert u(t)-b(t)-\mathcal{Q}%
(t)\Vert _{\mathcal{L}_{n}^{p}} &\leq &\sup_{[0,T_{1}\wedge \tau ^{\ast
}]}\Vert a-b\Vert _{\mathcal{L}_{n}^{p}}+\sup_{[0,T_{1}\wedge \tau ^{\ast
}]}\Vert \psi -\mathcal{Q}\Vert _{\mathcal{L}_{n}^{p}} \\
&\leq &C\sup_{[0,T_{1}\wedge \tau ^{\ast }]}\left\vert a-b\right\vert
+\sup_{[0,\tau ^{\ast }]}\Vert \psi -\mathcal{Q}\Vert _{\mathcal{L}_{n}^{p}}.
\end{eqnarray*}%
From (\ref{eq30}) and (\ref{e33}), we obtain 
\begin{equation*}
\sup_{t\in \lbrack 0,T_{1}\wedge \tau ^{\ast }]}\Vert u(t)-b(t)-\mathcal{Q}%
(t)\Vert _{\mathcal{L}_{n}^{p}}\leq C\varepsilon ^{1-(2m+1)\kappa }\;\ on\
\Omega ^{\ast }.
\end{equation*}%
Hence,%
\begin{equation*}
\mathbb{P}\Big( \sup_{t\in \lbrack 0,T_{1}\wedge \tau ^{\ast }]}\Vert
u(t)-b(t)-\mathcal{Q}(t)\Vert _{\mathcal{L}_{n}^{p}}>C\varepsilon
^{1-(2m+1)\kappa }\Big) \leq 1-\mathbb{P(}\Omega ^{\ast })\;.
\end{equation*}%
Using (\ref{e34}), yields (\ref{eq16}).
\end{proof}

\begin{proof}[Proof of Corollary \protect\ref{coll5}]
We note that by the semigroup estimate based on Assumptions \ref{asscoeff}
and Equation (\ref{E3})\ 
\begin{eqnarray*}
\Vert \psi (t)\Vert _{\mathcal{L}_{n}^{mp}} &\leq &\Vert e^{\varepsilon
^{-2}t\mathcal{A}}\psi (0)\Vert _{\mathcal{L}_{n}^{mp}}+\left\Vert \mathcal{Z%
}^{s}(t)\right\Vert _{\mathcal{L}_{n}^{mp}}+\frac{1}{\varepsilon ^{\rho }}%
\Big\|\int_{0}^{t}e^{\varepsilon ^{-2}\mathcal{A}_{s}(T-\tau )}\mathcal{F}%
^{s}(u)d\tau \Big\|_{\mathcal{L}_{n}^{mp}} \\
&\leq &e^{-\varepsilon ^{-2}t\omega }\left\Vert \psi (0)\right\Vert _{%
\mathcal{L}_{n}^{mp}}+\left\Vert \mathcal{Z}^{s}(t)\right\Vert _{\mathcal{L}%
_{n}^{mp}}+C\varepsilon ^{2-\rho }\sup_{\tau \in \left[ 0,\tau ^{\ast }%
\right] }(1+\left\Vert u\right\Vert _{\mathcal{L}_{n}^{pm}}^{m}),
\end{eqnarray*}
where we used Assumption \ref{Poly}. Thus by the definition of $\tau ^{\ast
} $ and the bounds on $\mathcal{Z}^{s}$\ (cf. (\ref{eq19a})) we obtain on $%
\Omega ^{\ast }$ 
\begin{equation*}
\sup_{t\in \lbrack 0,\tau ^{\ast }]}\left\Vert \psi (t)\right\Vert _{%
\mathcal{L}_{n}^{mp}}\leq C\varepsilon ^{-\kappa _{0}}.
\end{equation*}%
Thus from the Theorem \ref{th1}\ we derive%
\begin{equation*}
\Omega \supset \left\{ \tau ^{\ast }>T_{1}\right\} \supseteq
\{\sup_{[0,T_{1}\wedge \tau ^{\ast }]}\Vert u\Vert _{\mathcal{L}%
_{n}^{mp}}<\varepsilon ^{-\kappa }\}\supseteq \Omega ^{\ast }.
\end{equation*}
Hence, 
\begin{eqnarray*}
\sup_{t\in \lbrack 0,T_{1}]}\Vert u(t)-b(t)-\mathcal{Q}(t)\Vert _{\mathcal{L}%
_{n}^{mp}} &\leq &\sup_{[0,T_{1}]}\Vert a-b\Vert _{\mathcal{L}%
_{n}^{mp}}+\sup_{[0,T_{1}]}\Vert \psi -\mathcal{Q}\Vert _{\mathcal{L}%
_{n}^{mp}} \\
&\leq &C\sup_{[0,T_{1}]}\left\vert a-b\right\vert +\sup_{[0,\tau ^{\ast
}]}\Vert \psi -\mathcal{Q}\Vert _{\mathcal{L}_{n}^{mp}}.
\end{eqnarray*}%
Proceeding as in the proof of Theorem \ref{thm} we bound the error in $%
\mathcal{L}_{n}^{mp}.$
\end{proof}


\subsection{Application of Approximation Theorem I}

In this subsection we consider all examples with non-homogeneous Neumann
boundary condition on $[0,1]^{2}$. Here the eigenfunctions are%
\begin{equation*}
g_{k_{1},k_{2}}=\left\{ 
\begin{array}{ccl}
1 & \text{if} & k_{1}=k_{2}=0 \\ 
2\cos (\pi k_{1}x)\cos (\pi k_{2}y) & \text{if} & k_{1},k_{2}>0.%
\end{array}%
\right.
\end{equation*}%
The eigenvalues of the operator $-\mathcal{A}_{i}=-d_{i}(\partial
_{x}^{2}+\partial _{y}^{2})$ are $\lambda _{k_{1},k_{2}}=\pi
^{2}(k_{1}^{2}+k_{2}^{2})$. Define $f_{\ell }(z)$ as 
\begin{equation*}
f_{\ell }(z)=\left\{ 
\begin{array}{ccl}
1 & \text{if} & \ell =0 \\ 
\sqrt{2}\cos (\pi \ell z) & \text{if} & \ell >0.%
\end{array}%
\right.
\end{equation*}%
Now $g_{k}(x,y)=f_{k_{1}}(x)f_{k_{2}}(y)$ for $k\in \mathbb{N}_{0}^{2}$.


\subsubsection{Physical Application (Nonlinear Heat Eq.)}

The heat equation plays a significant role in several areas of science
including mathematics, probability theory and financial mathematics. In
probability theory for instance, the heat equation is used for studying
Brownian motion via the Fokker--Planck equation.

To apply our main Theorem \ref{thm}, we consider the following nonlinear
heat Equation with stochastic Neumann boundary condition. 
\begin{eqnarray}
\partial _{t}u &=&\varepsilon ^{-2}\left( \partial _{x}^{2}+\partial
_{y}^{2}\right) u+u-u^{3}\ \ \text{for }0\leq x\leq 1,\text{ }0\leq y\leq 1 
\notag \\
\partial _{x}u(t,x,0) &=&\sigma _{\epsilon }\partial _{t}W_{1}(t,x),\
\partial _{x}u(t,x,1)=\sigma _{\epsilon }\partial _{t}W_{2}(t,x)\ \text{for }%
x\in (0,1)\text{ \ \ \ \ }  \notag \\
\partial _{y}u(t,0,y) &=&\sigma _{\epsilon }\partial _{t}W_{3}(t,x),\
\partial _{y}u(t,1,y)=\sigma _{\epsilon }\partial _{t}W_{4}(t,y)\ \text{for }%
y\in (0,1).  \label{Heat}
\end{eqnarray}%
Define $W_{i}(t)$ for $i=1,2,3,4$ as $W_{i}(t)=\sum\limits_{j=1}^{\infty
}\alpha _{i,j}\beta _{i,j}(t)f_{j}$ and $\mathcal{N}=\{1\}$.

Our main Theorem \ref{thm} states that the solution of the nonlinear heat
equation (\ref{Heat}) with $\sigma _{\varepsilon }=\varepsilon ^{-1}$ is
well approximated by%
\begin{equation*}
u(t,x,y)=b(t)+\mathcal{Z}^{s}(t,x,y)+\mathcal{O}(\varepsilon ^{1-}),
\end{equation*}%
where $b$ is the solution of 
\begin{equation}
\partial _{t}b=(1-3C_{2})b-b^{3},  \label{Ampl}
\end{equation}%
and $C_{2}$ is a constant given by $C_{2}=\sum\limits_{k,j=1}^{\infty }\frac{%
\emph{q}_{k,j}}{\lambda _{k}+\lambda _{j}}P_{c}\left( g_{k}g_{j}\right) .$

We calculate%
\begin{equation*}
P_{c}\left( g_{k}g_{j}\right) =\left\{ 
\begin{array}{cl}
\frac{1}{2} & \text{if } k_{1}=j_{1},\text{ }k_{2}=j_{2} \\ 
0 & \text{otherwise,}%
\end{array}%
\right.
\end{equation*}%
and%
\begin{eqnarray*}
{q}_{k,j} &=&\delta _{k_{1},j_{1}}\alpha
_{1,k_{1}}^{2}f_{k_{2}}(0)f_{j_{2}}(0)+\delta _{k_{1},j_{1}}\alpha
_{2,k_{1}}^{2}f_{k_{2}}(1)f_{j_{2}}(1) \\
&&+\delta _{k_{2},j_{2}}\alpha _{3,k_{2}}^{2}f_{k_{1}}(0)f_{j_{1}}(0)+\delta
_{k_{2},j_{2}}\alpha _{4,k_{1}}^{2}f_{k_{1}}(1)f_{j_{1}}(1).
\end{eqnarray*}%
Thus%
\begin{equation*}
C_{2}=\frac{1}{2\pi ^{2}}\sum_{k_{1},k_{2}=1}^{\infty }\frac{1}{%
k_{1}^{2}+k_{2}^{2}}(\alpha _{1,k_{1}}^{2}+2\alpha _{2,k_{1}}^{2}+\alpha
_{3,k_{2}}^{2}+2\alpha _{4,k_{2}}^{2}).
\end{equation*}%
If we choose for any $\mu >0$ that $\alpha _{i,k}^{2}\leq C |k|^{-2\mu }$
for $i=1,\ldots,4$ and all $k\in\mathbb{N}$, then $C_{2}$ is finite and
furthermore, all summability conditions are satisfied.

Let us finally check the bound on $b$. Taking the product with $b$ on both
sides of (\ref{Ampl}), yields 
\begin{equation*}
\frac{1}{2}\partial _{t}|b| ^{2}=C|b| ^{2}-|b|^4 \leq C|b| ^{2}.
\end{equation*}%
Using Gronwall's lemma, we obtain for $0\leq t \leq T_0$ that%
\begin{equation*}
\sup_{[0,T_0]}|b|^{2}\leq |b(0)|^{2}e^{C T_0}.
\end{equation*}%
Thus Assumption \ref{Amp} is always true for deterministic initial
conditions if we choose $C_{0}$ sufficiently large.

\subsubsection{Chemical Application}

A simple archetypical example for a reaction-diffusion system is a cubic
auto-catalytic reaction between two chemicals according to the rule $%
A+B\rightarrow 2B$ with rate $r=\rho u_{1}u_{2}^{2}$.

Denoting by $u_{1}$ and $u_{2}$ the concentration of $A$ and $B$,
respectively. The two species satisfy the equations:%
\begin{equation}
\partial _{t}u_{1}=\frac{1}{\varepsilon ^{2}}\Delta u_{1}-\rho
u_{1}u_{2}^{2}\ \ \ \&\ \ \partial _{t}u_{2}=\frac{d}{\varepsilon ^{2}}%
\Delta u_{2}+\rho u_{1}u_{2}^{2}.  \label{Chim}
\end{equation}%
with respect to stochastic boundary conditions for $i=1,2$%
\begin{eqnarray}
\partial _{x}u_{i}(t,x,0) &=&\sigma _{\epsilon }\partial
_{t}W_{i_{1}}(t,x),\ \partial _{x}u_{i}(t,x,1)=\sigma _{\epsilon }\partial
_{t}W_{i_{2}}(t,x)\ \text{for }x\in (0,1)\text{ \ \ \ \ \ }  \notag \\
\partial _{y}u_{i}(t,0,y) &=&\sigma _{\epsilon }\partial
_{t}W_{i_{3}}(t,x),\ \partial _{y}u_{i}(t,1,y)=\sigma _{\epsilon }\partial
_{t}W_{i_{4}}(t,y)\ \text{for }y\in (0,1),  \label{Ch-bou}
\end{eqnarray}%
where $W_{i_{j}}(t)=\sum\limits_{k=1}^{\infty }\alpha _{i_{j},k}\beta
_{i_{j},k}(t)f_{k}$ for $j=1,\ldots,4,$ and $f_{k}$ defined as before.

\noindent We define $\mathcal{N}=\left\{ \left( 
\begin{array}{c}
1 \\ 
0%
\end{array}%
\right) ,\left( 
\begin{array}{c}
0 \\ 
1%
\end{array}%
\right) \right\} $ and take $\sigma _{\epsilon }=\varepsilon .$

\noindent Then our main theorem states that%
\begin{equation*}
u(t)=b(t)+\mathcal{Z}^{s}(t)+\mathcal{O}(\varepsilon ^{1-}),
\end{equation*}%
\begin{equation*}
\text{with }\quad u=\left( 
\begin{array}{c}
u_{1} \\ 
u_{2}%
\end{array}%
\right) ,\ \ b=\left( 
\begin{array}{c}
b_{1} \\ 
b_{2}%
\end{array}%
\right) \text{, and }\mathcal{Z}^{s}=\left( 
\begin{array}{c}
\mathcal{Z}_{1}^{s} \\ 
\mathcal{Z}_{2}^{s}%
\end{array}%
\right) ,
\end{equation*}%
where $b_{1}$ and $b_{2}$ are the solutions of%
\begin{eqnarray*}
\partial _{t}b_{1} &=&-\rho b_{1}b_{2}^{2}-\rho C_{2}b_{1} \\
\partial _{t}b_{2} &=&\rho b_{1}b_{2}^{2}+\rho C_{2}b_{1},
\end{eqnarray*}%
with%
\begin{equation*}
C_{2}=\frac{1}{2\pi ^{2}}\sum_{k_{1},k_{2}=1}^{\infty }\frac{1}{%
k_{1}^{2}+k_{2}^{2}}(\alpha _{2_{1},k_{1}}^{2}+2\alpha
_{2_{2},k_{2}}^{2}+2\alpha _{2_{3},k_{1}}^{2}+\alpha _{2_{4},k_{2}}^{2}).
\end{equation*}%
We note that high fluctuations in combination with fast diffusion lead to
effective new terms describing the transformation of $b_{1}$\ to $b_{2}$.
Although both terms individually do not change the average $\int
u_{i}dx=b_{i}$, their nonlinear combination does.

Let us check the bound on $b$ from Assumption \ref{Amp}. We note that%
\begin{equation*}
\sum_{i=1}^{2}\partial _{t}b_{i}=0 \quad\text{and thus }\quad
\sum_{i=1}^{2}b_{i}(t)=\sum_{i=1}^{2}b_{i}(0)=C_{0}.
\end{equation*}%
As $b_{1}(t)\geq 0$ and thus $b_{2}(t)\geq b_{2}(0)\geq 0,$ we have 
$
0\leq b_{i}(t)\leq \sum\limits_{i=1}^{2}b_{i}(t)\leq C_{0}.
$

\noindent 
Hence, for all times $t>0$ we obtain
$
\left\Vert b(t)\right\Vert =\Big( \sum\limits_{i=1}^{2}b_{i}^{2}(t)\Big)^{1/2}\leq
C_{0}\sqrt{2}.
$

\section{Proof of the Approximation Theorem II}

In this section, we use many ideas and lemmas of the previous sections, as
the main ideas are similar.

\begin{lemma}
\label{Lemma11}Let Assumption \ref{Poly} holds. Then
for $R^{(2)}$  defined in (\ref{E11b}) as%
\begin{equation*}
R^{(2)}_i=\sum_{|\ell| \geq 1}\int_{0}^{t}\frac{D^{\ell }\mathcal{F}_{i}(a)}{%
\ell !}P_{c}\mathcal{(}\varepsilon \psi )^{\ell }d\tau 
\end{equation*}
we have $R^{(2)}=\mathcal{O}(\varepsilon ^{1-m\kappa }).$
\end{lemma}

\begin{proof}
Using Assumption \ref{Poly}%
\begin{eqnarray*}
\mathbb{E}\sup_{[ 0,\tau ^{\ast }] }\vert R_i^{(2)}\vert^{p} &\leq
&C\sum_{|\ell| \geq 1}\frac{1}{\ell !} \mathbb{E}\sup_{[ 0,\tau ^{\ast
}]}\int_{0}^{t}|D^{\ell }\mathcal{F}_{i}(a)| \Vert \varepsilon \psi\Vert_{%
\mathcal{L}_{n}^{|\ell|}}^{|\ell|p} d\tau \\
&\leq &C\sum_{|\ell| \geq 1}\frac{1}{\ell !}[1+\varepsilon ^{(\ell
-m)p\kappa }]\varepsilon ^{\ell p(1-\kappa )} \\
&\leq &C\varepsilon ^{1-m\kappa }.
\end{eqnarray*}
\end{proof}


\begin{definition}
Define the set $\overset{\ast \ast }{\Omega }\subset \Omega $ such that for
sufficiently large $\zeta \gg 1$ all the following estimates hold on $%
\overset{\ast \ast }{\Omega }$ 
\begin{equation}
\sup_{\lbrack 0,\tau ^{\ast }]}\Vert \psi -\mathcal{Q}\Vert _{\mathcal{L}%
_{n}^{p}}<C\varepsilon ^{1-m\kappa -\kappa }\text{ },  \label{e39}
\end{equation}%
\begin{equation}
\sup_{\lbrack 0,\tau ^{\ast }]}\| \psi \| _{\mathcal{L}_{n}^{p}}<C%
\varepsilon ^{-\frac{3}{2}\kappa _{0}}\text{ },  \label{e40}
\end{equation}%
\begin{equation}
\sup_{\lbrack 0,\tau ^{\ast }]}| R^{(2)}|<C\varepsilon ^{1-m\kappa -\kappa }%
\text{ ,}  \label{e41}
\end{equation}%
and%
\begin{equation}
\sup_{\lbrack 0,T_{1}]}|b| ^{m-1}\leq \ln (\varepsilon ^{-\frac{1}{\zeta }})%
\text{.}  \label{e42}
\end{equation}
\end{definition}


\begin{proposition}
$\overset{\ast \ast }{\Omega }$ has approximately probability $1$.
\end{proposition}

\begin{proof}
\begin{equation*}
\mathbb{P}(\overset{\ast \ast }{\Omega })\geq 1-\mathbb{P}(\sup_{[0,\tau
^{\ast }]}\Vert \psi -\mathcal{Q}\Vert _{\mathcal{L}_{n}^{p}}\geq
C\varepsilon ^{1-m\kappa -\kappa })-\mathbb{P}(\sup_{[0,\tau ^{\ast }]}\Vert
\psi \Vert _{\mathcal{L}_{n}^{p}}\geq C\varepsilon ^{-\frac{3}{2}\kappa
_{0}})
\end{equation*}%
\begin{equation*}
-\mathbb{P}(\sup_{[0,\tau ^{\ast }]}| R^{(2)}|\geq C\varepsilon ^{1-m\kappa
-\kappa })-\mathbb{P}(\sup_{[0,T_{1}]}|b| ^{m-1}>\ln (\varepsilon ^{-\frac{1%
}{\zeta }})).
\end{equation*}%
Using Chebychev inequality and Lemmas \ref{Lemma1}, \ref{Lemma11} and
Corollary \ref{Cor1}, we obtain for $\kappa >\kappa _{0}$ and sufficiently
large $q>\frac{2p}{(\kappa -\kappa _{0})}>0$%
\begin{eqnarray}
\mathbb{P}(\overset{\ast \ast }{\Omega }) &\geq &1-C[\varepsilon ^{q\kappa
}+\varepsilon ^{\frac{1}{2}q\kappa }+\varepsilon ^{q(\kappa -\kappa _{0})}]-%
\mathbb{P}(\sup_{[0,T_{1}]}|b| ^{m-1}>\ln (\varepsilon ^{-\frac{1}{\zeta }}))
\notag \\
&\geq &1-C\varepsilon ^{\frac{1}{2}q(\kappa -\kappa _{0})}-\mathbb{P}%
(\sup_{[0,T_{1}]}|b| ^{m-1}>\ln (\varepsilon ^{-\frac{1}{\zeta }}))  \notag
\\
&\geq &1-C\varepsilon ^{\delta \kappa }\;,  \label{e56}
\end{eqnarray}%
where we used Assumption \ref{Amp1}.
\end{proof}

\begin{theorem}
\label{th2}Assume that Assumptions \ref{Wiener}, \ref{Poly} and \ref{Amp1}\
hold. Suppose $a(0)=\mathcal{O}(1)$ and $\psi (0)=\mathcal{O}(1)$. Let $b\in 
\mathcal{N}$ be a solution of (\ref{SODE})\ and $a\in \mathcal{N}$ as
defined in (\ref{Ampl2}). If the initial conditions satisfy $a(0)=b(0)$,
then for $\kappa <\frac{1}{m+2}$ we obtain 
\begin{equation}
\sup_{t\in \left[ 0,\tau ^{\ast }\wedge T_{1}\right] }\left\vert
a(t)-b(t)\right\vert \leq C\varepsilon ^{1-(m+2)\kappa }\text{\ \ \ on }%
\overset{\ast \ast }{\Omega }.  \label{e43}
\end{equation}
\end{theorem}

\begin{proof}
We follow the same steps as in the proof of Lemma \ref{th1} until Equation (%
\ref{e35a}) to obtain 
\begin{equation*}
\frac{1}{2}\partial _{t}|Q|^{2}\leq C\left(
1+|Q|^{m-1}+|R^{(2)}|^{m-1}+|b|^{m-1}\right) \left(
|Q|^{2}+|R^{(2)}|^{2}\right) .
\end{equation*}%
As long as $|Q|<1,$ using Equations (\ref{e41}) and (\ref{e42}), we obtain%
\begin{equation*}
\frac{1}{2}\partial _{t}\left\vert Q(t)\right\vert ^{2}\leq c(1+\ln
(\varepsilon ^{-\frac{1}{\zeta }}))\left\vert Q(t)\right\vert
^{2}+C\varepsilon ^{2-2(m+1)\kappa }\text{ \ on \ }\overset{\ast \ast }{%
\Omega }.
\end{equation*}%
Using Gronwall's lemma, we obtain for $t\leq \tau ^{\ast }\wedge T_{1}\leq
T_{0}$%
\begin{eqnarray*}
\left\vert Q(t)\right\vert ^{2} &\leq &C\varepsilon ^{2-2(m+1)\kappa }\exp
(2c(1+\ln (\varepsilon ^{-\frac{1}{\zeta }}))T_{0}) \\
&\leq &Ce^{2cT_{0}}\varepsilon ^{2-2(m+1)\kappa -2\tilde{\kappa}},
\end{eqnarray*}%
where $\tilde{\kappa}=\frac{cT_{0}}{\zeta }.$ If we choose $\tilde{\kappa}%
\leq \kappa $ for sufficiently large $\zeta ,$ then $\left\vert
Q(t)\right\vert <1$ for $\kappa <\frac{1}{m+2}$ and small $\varepsilon $.
Taking supremum on $[0,\tau ^{\ast }\wedge T_{1}]$%
\begin{equation}
\sup_{t\in \lbrack 0,\tau ^{\ast }\wedge T_{1}]}\left\vert Q(t)\right\vert
\leq C\varepsilon ^{1-(m+2)\kappa }\text{ \ on \ }\overset{\ast \ast }{%
\Omega }.  \label{e45}
\end{equation}%
Hence,%
\begin{eqnarray*}
\sup_{\lbrack 0,\tau ^{\ast }\wedge T_{1}]}\left\vert a-b\right\vert 
&=&\sup_{[0,\tau ^{\ast }\wedge T_{1}]}|Q-R^{(2)}|\leq \sup_{\lbrack 0,\tau
^{\ast }\wedge T_{1}]}|Q|+\sup_{[0,\tau ^{\ast }\wedge T_{1}]}|R^{(2)}| \\
&\leq &C\varepsilon ^{1-(m+2)\kappa }\text{ \ on \ }\overset{\ast \ast }{%
\Omega }.
\end{eqnarray*}
\end{proof}

Now we can use the results obtained previously to prove the main result of
Theorem \ref{thm2} and Corollary \ref{coll5a} \ for the SPDE (\ref{eq1a}).

\noindent
\begin{proof}[Proof of Theorem \protect\ref{thm2}]
\!\! Similar steps than the proof of Theorem \ref{thm}.
\end{proof}

\noindent
\begin{proof}[Proof of Corollary \protect\ref{coll5a}]
\!\! Similar steps than the proof of Corollary \ref{coll5}.
\end{proof}


\subsection{Application of Approximation Theorem II}

In this subsection we apply our main Theorem \ref{thm2} to the nonlinear
heat equation (\ref{Heat}) and a cubic auto-catalytic reaction (\ref{Chim})
with $\sigma _{\varepsilon }=1$ and non-zero $\alpha _{k,0}$.

\subsubsection{Physical Application (Nonlinear Heat Eq.)}

Our main Theorem \ref{thm2} in this case states that the solution of (\ref%
{Heat}) takes the form 
\begin{equation*}
u(t)=b(t)+\mathcal{O}(\varepsilon ^{1-}),
\end{equation*}%
where $b$ is the solution of stochastic ordinary differential equation%
\begin{equation}
db=[b-b^{3}]dt+dB,  \label{e50}
\end{equation}%
and $B$\ is a $\mathbb{R}$-valued standard Brownian motion given by%
\begin{equation*}
B(t)=\alpha _{1,0}\beta _{1,0}(t)+\alpha _{2,0}\beta _{2,0}(t)+\alpha
_{3,0}\beta _{3,0}(t)+\alpha _{4,0}\beta _{4,0}(t).
\end{equation*}

\noindent To check the bound on $b$ consider $\exp \{\delta |b|^{2}\}$. We
note that%
\begin{equation}
d\exp \{\delta |b|^{2}\}=\delta \exp \{\delta |b|^{2}\}d|b|^{2}+\delta
^{2}\exp \{\delta |b|^{2}\}(d|b|^{2})^{2},  \label{e51}
\end{equation}%
and%
\begin{equation*}
d|b|^{2}=2b\cdot db+db\cdot db.
\end{equation*}%
From (\ref{e50}) we obtain for some constant $c>0$ 
\begin{eqnarray*}
d|b|^{2} &=&2|b|^{2}dt-2|b|^{4}dt+2b\cdot dB+dB\cdot dB \\
&=&(C+2|b|^{2}-2|b|^{4})dt+2b\cdot dB.
\end{eqnarray*}%
Substituting this into (\ref{e51}), yields%
\begin{eqnarray}
d\exp \{\delta |b|^{2}\} &=&\delta (C+(2+4\delta )|b|^{2}-2|b|^{4})\exp
\{\delta |b|^{2}\}dt+2\delta \exp \{\delta |b|^{2}\}b\cdot dB 
\nonumber\\
&\leq &c_{\delta }\exp \{\delta |b|^{2}\}dt+2\delta \exp \{\delta
|b|^{2}\}b\cdot dB.\label{e52}
\end{eqnarray}

Integrating from $0$ to $t$ and taking expectation, yields%
\begin{equation*}
\mathbb{E}\exp \{\delta |b(t)|^{2}\}\leq \mathbb{E}\exp \{\delta
|b(0)|^{2}\}+c_{\delta }\int_{0}^{t}\mathbb{E}\exp \{\delta |b|^{2}\}dt.
\end{equation*}%
As $\mathbb{E}\exp \{3\delta |b(0)|^{2}\}\leq C$ and applying Gronwall's
lemma, yields for $t\leq T_{1}$%
\begin{equation}
\sup_{\lbrack 0,T_{1}]}\mathbb{E}\exp \{\delta \left\vert b\right\vert
^{2}\}\leq C.  \label{e53}
\end{equation}%
With $3\delta $ instead of $\delta $, we have%
\begin{equation}
\sup_{\lbrack 0,T_{1}]}\mathbb{E}\exp \{3\delta \left\vert b\right\vert
^{2}\}\leq C.  \label{e53a}
\end{equation}%
Taking expectation after supremum on both sides of (\ref{e52}) to obtain%
\begin{eqnarray*}
\lefteqn{\mathbb{E}\sup_{t\in \lbrack 0,T_{1}]}\exp \{\delta |b(t)|^{2}\})}
\\
&\leq &\mathbb{E}\exp \{\delta |b(0)|^{2}\}+c_{\delta }\mathbb{E}\sup_{t\in
\lbrack 0,T_{1}]}\int_{0}^{t}\exp \{\delta |b(s)|^{2}\}ds \\
&&\qquad +2\delta \mathbb{E}\sup_{t\in \lbrack 0,T_{1}]}\int_{0}^{t}b(s)\exp
\{\delta |b(s)|^{2}\}dB(s) \\
&\leq &C+c_{\delta }\mathbb{E}\int_{0}^{T_{1}}\exp \{\delta
|b(s)|^{2}\}ds+2\delta \mathbb{E}\Big(\int_{0}^{T_{1}}b(s)^{2}\exp \{2\delta
|b(s)|^{2}\}ds\Big)^{1/2}.
\end{eqnarray*}%
Using (\ref{e53a}) together with $xe^{2\delta x}\leq Ce^{3\delta x}$ for all 
$x>0$, yields%
\begin{equation*}
\mathbb{E}\sup_{t\in \lbrack 0,T_{1}]}\exp \{\delta |b(t)|^{2}\}\leq C.
\end{equation*}%
Now, using Chebychev inequality%
\begin{equation*}
\mathbb{P}(\sup_{[0,T_{1}]}|b(t)|^{2}>\ln (\varepsilon ^{-\kappa }))\leq 
\frac{\mathbb{E}\sup_{t\in \lbrack 0,T_{1}]}\exp \left( \delta
|b(t)|^{2}\right) }{\exp \left( \delta \ln (\varepsilon ^{-\kappa })\right) }%
\leq C\varepsilon ^{\delta \kappa }.
\end{equation*}

\subsubsection{Chemical Application}

Our main theorem states that the solution of (\ref{Chim}) takes the form%
\begin{equation*}
u(t)=b(t)+\mathcal{O}(\varepsilon ^{1-}),
\end{equation*}
\begin{equation*}
\text{ with}\quad u=\left( 
\begin{array}{c}
u_{1} \\ 
u_{2}%
\end{array}%
\right) \ \text{ and }b=\left( 
\begin{array}{c}
b_{1} \\ 
b_{2}%
\end{array}%
\right) .
\end{equation*}%
In this case $b_{1}$ and $b_{2}$ are the solutions of%
\begin{equation}
db_{1}=-\rho b_{1}b_{2}^{2}dt+dB_{1}(t)\text{ \ \& \ }db_{2}=\rho
b_{1}b_{2}^{2}dt+dB_{2}(t),  \label{E50}
\end{equation}%
where%
\begin{equation*}
B_{i}(t)=\alpha _{i_{1},0}\beta _{i_{1},0}+\alpha _{i_{2},0}\beta
_{i_{2},0}(t)+\alpha _{i_{3},0}\beta _{i_{3},0}(t)+\alpha _{i_{4},0}\beta
_{i_{4},0}(t)\text{ \ for }i=1,2.
\end{equation*}
To verify the bound on $b$ define first the stopping $T_{1}$ as%
\begin{equation*}
T_{1}=T_{0}\wedge \inf \left\{ t>0:\exists \text{ }i\in \left\{ 1,2\right\}
:b_{i}(t)<0\right\} .
\end{equation*}%
This means that our approximation result is only true as long as the
concentrations $b_{i}$ are non-negative.

Now, we note that%
\begin{equation*}
\sum_{i=1}^{2}db_{i}=\sum_{i=1}^{2}dB_{i}.
\end{equation*}%
Integrating from $0$ to $t$, yields%
\begin{equation}
\sum_{i=1}^{2}b_{i}(t)=\sum_{i=1}^{2}b_{i}(0)+\sum_{i=1}^{2}B_{i}(t).
\label{e55}
\end{equation}%
Hence, up to $T_{1}$ we obtain%
\begin{equation*}
|b(t)| \leq
\sum_{i=1}^{2}b_{i}(t)=\sum_{i=1}^{2}b_{i}(0)+\sum_{i=1}^{2}B_{i}(t)\leq 
\sqrt{2}| B(t)| +\sqrt{2}|b(0)| ,
\end{equation*}%
where we used $\left( x^{2}+y^{2}\right) ^{1/2}\leq | x | +| y| \leq \sqrt{2}
( x^{2}+y^{2} ) ^{1/2}.$ Moreover,%
\begin{equation*}
|b(t)| ^{2}\leq 4| B(t)| ^{2}+4| b(0)|^{2}.
\end{equation*}%
Thus%
\begin{equation*}
\mathbb{E}\sup_{[ 0,T_{1}] }\exp \{ \delta |b| ^{2}\} \leq \mathbb{E}\sup_{[
0,T_{1}] }\exp \{ 4\delta |B| ^{2}\} \cdot 
\exp \{ 4\delta |b(0)| ^{2}\} \leq C, 
\end{equation*}
but only for sufficiently small $\delta .$ Using Chebychev inequality%
\begin{equation*}
\mathbb{P}(\sup_{[0,T_{1}]}|b(t)| ^{2}>\ln (\varepsilon ^{-\kappa }))\leq 
\frac{\mathbb{E}\sup_{t\in [ 0,T_{1}] }\left( \exp \left( \delta |b(t)|
^{2}\right) \right) }{\exp \left( \delta \ln (\varepsilon ^{-\kappa
})\right) }\leq C\varepsilon ^{\delta \kappa }.
\end{equation*}
So the probability is close, but not very close to $1$, as $\delta$ cannot
be arbitrarily large. %
%
%
%
%
%
%
%
%
%
%
%
%
%
%
%
%
%
%
%
%
%
%

\section*{Acknowledgements}

This work was supported by the Deutsche Forschungsgemeinschaft (DFG)
`Multiscale Analysis of SPDEs` (DFG BL535/9-2).

\end{document}